\documentclass[12pt,reqno]{amsart}

\usepackage{amsmath}
\usepackage{amsthm}
\usepackage{amssymb,amsfonts}

\usepackage{enumerate}
\usepackage{enumitem}
\usepackage{mathrsfs}
\usepackage{graphicx}
\usepackage{bm}
\usepackage{microtype}

\usepackage[colorlinks,linkcolor=blue,
anchorcolor=blue,citecolor=blue]{hyperref}

\usepackage[numbers,sort&compress]{natbib}

\usepackage{aliascnt}
\usepackage[nameinlink,noabbrev]{cleveref}

\usepackage{marvosym}
\usepackage{epstopdf}
\usepackage{caption}
\usepackage{multicol}
\usepackage{multirow}
\usepackage{makebox}
\usepackage{subfigure}
\usepackage{dsfont}

\IfFileExists{bbm.sty}{\usepackage{bbm}}{\let\mathbbm\boldsymbol}

\allowdisplaybreaks

\numberwithin{equation}{section}

\newtheorem{theorem}{Theorem}[section]

\newaliascnt{lemma}{theorem}
\newtheorem{lemma}[lemma]{Lemma}
\aliascntresetthe{lemma}

\newaliascnt{proposition}{theorem}

\aliascntresetthe{proposition}

\newaliascnt{corollary}{theorem}

\aliascntresetthe{corollary}

\theoremstyle{definition}

\newaliascnt{remark}{theorem}

\aliascntresetthe{remark}

\newaliascnt{definition}{theorem}
\newtheorem{definition}[definition]{Definition}
\aliascntresetthe{definition}

\newcommand{\R}{\mathbb R}

\newcommand{\Sym}{\operatorname{Sym}}
\newcommand{\SL}{\operatorname{SL}}
\newcommand{\GL}{\operatorname{GL}}

\title{$\mathrm{SL}(n)$-Covariant Symmetric Tensor Valuations on Polytopes}

\author[Zeng and Cheng]{Chunna Zeng$^{*}$ and Xiaoyu Cheng}

\address{1. School of Mathematical Sciences,
Chongqing Normal University,
Chongqing 401331,
People's Republic of China}

\email{zengchn@163.com}

\address{2. School of Mathematics and Sciences,
Chongqing Normal University,
Chongqing 401331,
People's Republic of China}

\email{15609089289@163.com}

\thanks{The first author is supported by National Natural Science Foundation of China
(Grant No. 12141101) and the Science and Technology Research Program of
Chongqing Municipal Education Commission
(Grant No. KJZD-K202500503)}


\thanks{{*}Corresponding author: Chunna Zeng}

\begin{document}

\maketitle

\begin{abstract}

All $\operatorname{SL}(n)$-covariant symmetric tensor valuations on
convex polytopes in $\mathbb R^n$  are completely
classified without any continuity assumptions. It is shown that the only such valuation is the moment tensor.
\end{abstract}

\section{Introduction}
The study and classification of geometric notions compatible with
transformation groups are important tasks in geometry, as proposed in Felix
Klein's Erlangen program in 1872.
Symmetric tensors provide a common framework for scalar, vector, and some higher-order geometric quantities. Among the fundamental examples are the moment tensors and LYZ tensors (see \cite{Ludwig2023,Ludwig2002Moment,Ludwig2002Origin,Ludwig2003,HP14,HP16,HP17}). 
Denote by $\mathcal P_{(o)}^n$ the space of convex polytopes containing the origin in their interiors.  For a positive integer $p$, the $p$-th moment tensor and the $p$-th LYZ tensor of $P\in\mathcal P_{(o)}^n$ are defined by
\[
M^{p,0}(P)=\int_P x^{\odot p}\,dx,
\qquad
M^{0,p}(P)=\int_{S^{n-1}} y^{\odot p}\,dS_{n-1,p}(P,y),
\]
respectively. Here, $x^{\odot p}$ denotes the $p$-fold symmetric tensor product of $x\in\mathbb R^n$
(see Section \ref{sec2:symmetric-tensors} for details),
and $S_{n-1, p}(P, \cdot)$ denotes the $L_p$ surface
area measure of $P$. For $p=1$, $M^{1,0}(P)$ is the moment vector, which equals the centroid of $P$
multiplied by its volume, whereas $M^{0,1}(P)$ vanishes. For $p=2$, the moment matrix $M^{2,0}(P)$ is
associated with the classical Legendre ellipsoid,
while the LYZ matrix $M^{0,2}(P)$ gives rise to the LYZ ellipsoid introduced by Lutwak,
Yang, and Zhang \cite{LutwakYangZhang2000Ellipsoid}. 
For even integers $p\ge2$, the tensors $M^{p,0}(P)$
and $M^{0,p}(P)$ determine the  $L_p$ moment
body and $L_p$ projection body, respectively.
These bodies are fundamental objects in the $L_p$
Brunn--Minkowski theory. For more information, see 
\cite{Ludwig2002Moment,Lutwak1996,LutwakYangZhang2000Affine,
HaberlSchuster2009,Ludwig2005,Haberl2012,
Parapatits,SchusterWannerer2012,Wannerer2011,Ludwig2003,MilmanPajor1989}. 
%This naturally leads to the problem of classifying  symmetric tensor-valued maps that are compatible
%with linear transformations.

Both moment tensors and LYZ tensors satisfy the inclusion–exclusion principle and are fundamental examples of tensor valuations. Recall that a map 
$Z:\mathcal S\to\langle \mathcal{A},+\rangle$ defined on a collection of sets $\mathcal S$ with values in an abelian
semigroup $\langle \mathcal{A},+\rangle$ 
is called a valuation if
\[
Z(P_1)+Z(P_2)=Z(P_1\cup P_2)+Z(P_1\cap P_2)
\]
whenever $P_1$, $P_2$, $P_1\cap P_2$, $P_1\cup P_2\in \mathcal S$.
The origins of valuation theory can be traced back to Dehn’s solution of Hilbert’s third problem at the beginning of the twentieth century. Nearly half a century later, Hadwiger initiated a systematic study of valuations with his celebrated characterization theorem \cite{Hadwiger1957},
 % A landmark in its development
%is Hadwiger’s characterization theorem \cite{Hadwiger1957}, 
which states that every continuous, rigid-motion invariant, real-valued valuation on convex bodies (i.e., nonempty compact convex sets) in $\mathbb R^n$ is a linear combination of the intrinsic volumes. This theorem provides an effective tool for proving many important results in integral geometry and geometric probability. Valuation theory has also developed close connections 
with integral geometry, geometric measure theory, and differential geometry (see \cite{Alesker2010Manifolds,AleskerBernig2017, AleskerFu1,JensenKiderlen2017,Klain1,LudwigMussnig2023}).
In recent decades, Ludwig has made significant contributions to its development (see \cite{Ludwig2002Origin,Ludwig2002Projection,Ludwig2003,Ludwig2002Moment,LudwigMussnig2023,LudwigReitzner2017,LudwigSilverstein2017,LudwigReitzner2010}). 

Tensor valuations take values in finite-dimensional tensor spaces.
The study of tensor valuations was initiated by McMullen
\cite{McMullen1997} and attracted increasing attention following
Alesker's breakthrough \cite{Alesker1999Rotation}. The techniques
developed in \cite{Alesker1999Rotation} enabled Alesker to obtain
a long-sought-after characterization of the Minkowski tensors
compatible with rigid motions \cite{Alesker1999Isometry}.
Tensor valuations have since been studied extensively (see
\cite{AbardiaEtAl2019,BernigHug2018,BoroczkyDomokosSolanes2021,
HugSchneider2014,Ludwig2003,LudwigSilverstein2017,LiMa2025}).
%For characterizations of Minkowski tensors compatible
%with rigid motions and related results, see
%\cite{Alesker1999Rotation,Alesker1999Isometry,
%BernigHug2018,BoroczkyDomokosSolanes2021,
%HugSchneider2014,McMullen1997}.

This naturally leads to the problem of classifying tensor-valued valuations compatible with linear transformations. This classification problem lies within affine valuation theory. Studying geometric functionals compatible with affine transformations is one of the central topics in convex geometry; see, e.g., \cite{Haberl2012,HP14,Ludwig2005,LudwigReitzner2010,
SchusterWannerer2012,HenkelWannerer2025}.
For symmetric tensors of order two,
Ludwig \cite{Ludwig2003} established classification results for 
measurable matrix-valued valuations compatible with the general 
linear group $\mathrm{GL}(n)$, characterizing the moment matrix 
$M^{2,0}$ and the LYZ matrix $M^{0,2}$.
%Moment tensors also play an important role in affine tensor valuation theory. 
Building on the pioneering work of Ludwig \cite{Ludwig2002Moment,Ludwig2002Origin,Ludwig2003}
and Ludwig and Reitzner \cite{LudwigReitzner2010}, Haberl and Parapatits \cite{HP14,HP16,HP17}
investigated symmetric tensor valuations  in the centro-affine setting.
%established a classification of $\SL(n)$-covariant
%symmetric tensor valuations in the centro-affine setting.  
%Besides $M^{p,0}$, the centro-affine classification
%involves  $M^{0,p}$ associated with the cone-volume measure (see \cite{HP17}). 

Let $\mathcal P^n$ denote the space of convex polytopes in 
$\mathbb R^n$ and $\mathcal P_o^n$ the subset consisting of 
polytopes containing the origin. For $n\geq 2$ and an integer $p\geq 0$, let 
$(\mathbb R^n)^{\otimes p}$ denote the $p$-fold tensor product 
of $\mathbb R^n$ and $\operatorname{Sym}^p(\mathbb R^n)$ its 
subspace of symmetric tensors, with 
$(\mathbb R^n)^{\otimes 0}=\operatorname{Sym}^0(\mathbb R^n)=\mathbb R$. Denote by $\phi\cdot \mathbbm{t}$ the natural action of
$\phi\in\operatorname{GL}(n)$ on $\mathbbm{t}\in(\mathbb R^n)^{\otimes p}$
(see section~\ref{sec2:symmetric-tensors}).
%The natural action of $\operatorname{GL}(n)$ on 
%$(\mathbb R^n)^{\otimes p}$ is given by
%\[
%\phi\cdot \mathbbm{t}=\phi^{\otimes p}\mathbbm{t}
%\]
%for all $\phi\in\operatorname{GL}(n)$ and 
%$\mathbbm{t}\in(\mathbb R^n)^{\otimes p}$.
A map $Z:\mathcal F\to\operatorname{Sym}^p(\mathbb R^n)$ 
is called $\operatorname{SL}(n)$-covariant if
\[
Z(\phi P)=\phi\cdot Z(P)
\]
for every $P\in\mathcal F$ and $\phi\in\operatorname{SL}(n)$, 
where $\mathcal F$ is one of $\mathcal P^n_{(o)}$, 
$\mathcal P_o^n$, or $\mathcal P^n$.
%, and $\phi\cdot Z(P)$ denotes the natural action of $\phi$ on $Z(P)$ (see section~\ref{sec2:symmetric-tensors}).

%For $n\geq2$ and an integer \(p\geq 0\), let $(\mathbb R^n)^{\otimes p}$ denote the $p$-fold tensor product of $\mathbb R^n$ and $\Sym^p(\mathbb R^n)$ its subspace of symmetric tensors, with $(\mathbb R^n)^{\otimes 0} =\Sym^0(\mathbb R^n)=\mathbb R$. Denote by  $\mathcal P^n$ and $\mathcal P_o^n$ the space of convex polytopes  and its subspace consisting of 
%polytopes containing the origin in  $\mathbb R^n$, respectively.
%Denote by $\mathcal P^n$ the space of convex polytopes in $\mathbb R^n$, and by $\mathcal P_o^n$ the subspace of polytopes containing the origin.
%A map 
%$Z:\mathcal  F\to\operatorname{Sym}^p(\mathbb R^n)$
%is called 
%$\operatorname{SL}(n)$-covariant if
%\[
%Z(\phi P)=\phi^{\otimes p}Z(P)
%\]
%for every $P\in\mathcal F$ and $\phi\in\operatorname{SL}(n)$, where $\mathcal F$ is one of $\mathcal P^n_{(o)}$, 
%$\mathcal P_o^n$, or $\mathcal P^n$.

%A map $Z:\mathcal P^n\to\operatorname{Sym}^p(\mathbb R^n)$ 
%is called $\operatorname{SL}(n)$-covariant if
%\[
%Z(\phi P)=\phi^{\otimes p}Z(P)
%\]
%for every $P\in\mathcal P^n$ and $\phi\in\operatorname{SL}(n)$, where $\mathcal P^n$ denotes the space of convex polytopes in $\mathbb R^n$.

The following is the classification theorem of Haberl and Parapatits \cite{HP17} for measurable $\SL(n)$-covariant symmetric tensor valuations
on $\mathcal P^n_{(o)}$. 
%Haberl and Parapatits \cite{HP17} obtained a complete  classification of measurable $\operatorname{SL}(n)$-covariant symmetric tensor valuations on $\mathcal P^n_{(o)}$. Their classification takes the following form.

\begin{theorem}[\cite{HP17}]
Let $p\ge2$ and $n\ge3$. A map
$Z:\mathcal P^n_{(o)}\to\Sym^p(\mathbb R^n)$
is a measurable $\SL(n)$-covariant valuation if and only if there exist
constants $a_1,a_2\in\mathbb R$ such that
\[
Z(P)
=
a_1M^{p,0}(P)+a_2M^{0,p}(P^*)
\]
for all $P\in\mathcal P^n_{(o)}$, where
$P^*$ denotes the polar body of $P$.
\end{theorem}

%Let $\mathcal P_o^n$ be the subset of polytopes that contain the origin.
The results of Haberl and Parapatits
\cite{HP14,HP16,HP17} were  extended to valuations on
$\mathcal P_o^n$ and $\mathcal P^n$ without any regularity assumptions. 
For $p=0$, Ludwig and Reitzner \cite{LudwigReitzner2017} established a complete  classification of $\SL(n)$-invariant real-valued valuations on $\mathcal P_o^n$ and $\mathcal P^n$.
The result was subsequently extended to vector-valued valuations by Zeng and Ma \cite{ZM18}, who classified $\operatorname{SL}(n)$-covariant vector-valued valuations. Furthermore, for $p=2$,  Ma \cite{Ma21} established a classification of $\operatorname{SL}(n)$-covariant symmetric matrix-valued valuations.

%The results of Haberl and Parapatits \cite{HP14,HP16,HP17} were  subsequently extended to valuations on $\mathcal P_o^n$ and $\mathcal P^n$ without any regularity assumptions. For $p=0$, Ludwig and Reitzner \cite{LudwigReitzner2017} established a complete  classification of $\SL(n)$-invariant real-valued valuations on $\mathcal P_o^n$ and $\mathcal P^n$ without continuity assumptions. For $p=1$,  Zeng and Ma \cite{ZM18} obtained the classification of $\SL(n)$-covariant vector-valued valuations. Furthermore, for $p=2$, Ma \cite{Ma21} established a classification of \(\operatorname{SL}(n)\)-covariant symmetric matrix-valued valuations.

Motivated by above results, we aim to extend the work of Haberl and Parapatits
\cite{HP17} to symmetric tensor valuations on $\mathcal P_o^n$ and $\mathcal P^n$
without regularity assumptions. Specifically, we enlarge the domain from
$\mathcal P_{(o)}^n$ to $\mathcal P_o^n$ and
$\mathcal P^n.$  More significantly, our classification does not require the measurability
assumption.
%This problem belongs to the broader study of geometric
%functionals compatible with affine transformations, an important theme
%in convex geometry (see
%\cite{Haberl2012,HP14,Ludwig2005,LudwigReitzner2010,
%SchusterWannerer2012,HenkelWannerer2025}).
Note that the results of Haberl and Parapatits 
\cite{HP14,HP16,HP17} cannot follow from our results or from the aforementioned works---at least we do not know how to do so.
%Note that our do not, however, imply
%the classification of Haberl and Parapatits \cite{HP14,HP16,HP17}.
Our first main result gives a complete classification
on polytopes containing the origin.

\begin{theorem}\label{thm:main}
Let \(n\ge3\) and let \(p\ge1\) be an integer. A map $Z:\mathcal P_o^n\longrightarrow\operatorname{Sym}^p(\mathbb R^n)$
is an \(\operatorname{SL}(n)\)-covariant valuation if and only if there exists a constant \(c\in\mathbb R\) such that
\begin{equation}\label{eq:main-class}
 Z(P)=c\,M^{p,0}(P)
\end{equation}
for every \(P\in\mathcal P_o^n\).
\end{theorem}

Next, we consider the space of all convex polytopes $\mathcal P^n$. This step is as in the
classification of convex body valued valuations by Schuster and Wannerer \cite{SchusterWannerer2012} and
Wannerer \cite{Wannerer2011}.
\begin{theorem}\label{thm:all}
Let \(n\ge3\) and let \(p\ge1\) be an integer. A map $Z:\mathcal P^n\longrightarrow\operatorname{Sym}^p(\mathbb R^n)$
is an \(\operatorname{SL}(n)\)-covariant valuation if and only if there exist constants \(c_1,c_2\in\mathbb R\) such that
\begin{equation}\label{eq:all-class}
Z(P)=c_1M^{p,0}(P)+c_2M^{p,0}([o,P])
\end{equation}
for every \(P\in\mathcal P^n\), where $[o,P]$ denotes the convex hull of $P$ and the origin.
\end{theorem}

\section{Preliminaries and notation}

\subsection{Valuations and the decomposition of $\SL(n)$-covariant valuations}
\label{sec2:polytopes-valuations}

 In this paper, we work in $n$-dimensional Euclidean space $\mathbb R^n$, with $n\ge3$. Let $o$ denote  the origin  and $e_1,\ldots,e_n$ the standard basis vectors.
 %We work in $n$-dimensional Euclidean space $\mathbb R^n$, where
%$n\ge3$. The origin and standard basis are denoted by $o$ and
%$e_1,\ldots,e_n$, respectively. 
For $P\in\mathcal P^n$, denote by
$\operatorname{aff}P$, $\operatorname{relint}P$, and $\dim P$
its affine hull, relative interior, and dimension,
respectively. For $v_1,\ldots,v_k\in\mathbb R^n$, let  
$[v_1,\ldots,v_k]$ denote their convex hull.
Let $\mathcal Q^n$ stand for either $\mathcal P^n$ or
$\mathcal P_o^n$. By convention, $Z(\varnothing)=\boldsymbol{0}$,
where $\boldsymbol{0}$ denotes the zero tensor.
A valuation on $\mathcal Q^n$ is called \emph{simple}
if it vanishes on all polytopes of dimension less than $n$. 
The following inclusion---exclusion principle will be used repeatedly;
see \cite{Klain1,Parapatits}.

\begin{lemma}\label{lemma2.1}
Let $\mathcal A$ be an abelian group and let
$Z:\mathcal Q^n\to\mathcal A$ be a valuation, with
$Z(\varnothing)=0$. Then
\[
Z(P_1\cup\cdots\cup P_m)
=
\sum_{\varnothing\ne I\subseteq\{1,\ldots,m\}}
(-1)^{|I|-1}
Z\left(\bigcap_{i\in I}P_i\right)
\]
for all $P_1,\ldots,P_m\in\mathcal Q^n$ such that
$P_1\cup\cdots\cup P_m\in\mathcal Q^n$.
\end{lemma}

 \emph{A triangulation} of a $k$-dimensional polytope
$P$ is a finite family of $k$-dimensional simplices
$\{T_1,\ldots,T_r\}$ with pairwise disjoint relative interiors,
such that $P=\bigcup_{i=1}^r T_i$ and every nonempty intersection
$T_{i_1}\cap\cdots\cap T_{i_j}$,
$1\le i_1<\cdots<i_j\le r$, is a simplex. Therefore, we can apply the inclusion--exclusion principle; see \cite{Parapatits}.

%\subsection{Decomposition of $\SL(n)$-covariant valuations}\label{sec2:decomposition}

%Fix an integer $p\ge1$. 
%A valuation
%$Z:\mathcal Q^n\to\operatorname{Sym}^p(\mathbb R^n)$ is
%$\SL(n)$-covariant if
%$Z(\phi P)=\phi\cdot Z(P)$
%for all $P\in\mathcal Q^n$ and $\phi\in\SL(n)$.
Let $\SL^{\pm}(n)=\{\phi\in\GL(n):\det\phi=\pm1\}.$
For $\varepsilon\in\{0,1\}$, a valuation $Z:\mathcal Q^n\to\operatorname{Sym}^p(\mathbb R^n)$
is called $\operatorname{SL}^{\pm}(n)$-$\varepsilon$-covariant if
\[
Z(\phi P)=(\det\phi)^\varepsilon\,\phi\cdot Z(P)
\]
for all $P\in\mathcal Q^n$ and $\phi\in\SL^{\pm}(n)$.

Following Haberl and Parapatits \cite{HP17}, any $\operatorname{SL}(n)$-covariant valuation $Z$ admits
the following decomposition. Fix
$\theta\in\SL^{\pm}(n)\setminus\SL(n)$ and define
\begin{equation}\label{eq:parity}
Z^\varepsilon(P)
=\frac12\bigl(Z(P)+(-1)^\varepsilon\,
\theta\cdot Z(\theta^{-1}P)\bigr),
\qquad \varepsilon\in\{0,1\}.
\end{equation}
The $\SL(n)$-covariance of $Z$ ensures that
$Z^\varepsilon$ is independent of $\theta$ and satisfies
\begin{equation}\label{eq:epsilon-cov}
Z^\varepsilon(\phi P)
=(\det\phi)^\varepsilon\,\phi\cdot Z^\varepsilon(P)
\end{equation}
for all $P\in\mathcal Q^n$ and $\phi\in\SL^{\pm}(n)$.
Thus, $Z=Z^0+Z^1$ is the unique decomposition of $Z$ into valuations
of these two types. If $Z$ is simple, so are $Z^0$ and $Z^1$.

\subsection{The dissections of the standard simplex}
\label{sec:standard-simplices}
For $1\leq k\leq n$, let
\[
T^k=[o,e_1,\ldots,e_k],\qquad
\widetilde T^{k-1}=[e_1,\ldots,e_k]
\]
denote the $k$-dimensional standard simplex and its facet opposite the origin, respectively, with the convention
$T^0=\{o\}$. Denote by $\mathcal T^k$ the set of
$k$-dimensional simplices with a vertex at the origin
and by $\widetilde{\mathcal T}^{k-1}$ the set of
$(k-1)$-dimensional simplices $T\subset\mathbb R^n$
with $o\notin\operatorname{aff}T$.
Every $T\in\widetilde{\mathcal T}^{k-1}$ can be written as
\begin{equation}\label{eq:tilde-simplex-normalize}
T=\phi(s\widetilde T^{k-1})
\end{equation}
for some $s>0$ and $\phi\in\operatorname{SL}(n)$.
%For $1\le k\le n$, denote by
%$T^k=[o,e_1,\ldots,e_k]$ the $k$-dimensional standard simplex
%and by $\widetilde T^{k-1}=[e_1,\ldots,e_k]$ a
%$(k-1)$-dimensional simplex, with the convention that %$T^0=\{o\}$.
%Let $\mathcal T^k$ be the set of $k$-dimensional simplices with a vertex
%at the origin, and $\widetilde{\mathcal T}^{k-1}$ the set of
%$(k-1)$-dimensional simplices $T\subset\mathbb R^n$ with
%$o\notin\operatorname{aff}T$.
%Every $T\in\widetilde{\mathcal T}^{k-1}$ has the form
%\begin{equation}\label{eq:tilde-simplex-normalize}
%T=\phi(s\widetilde T^{k-1})
%\end{equation}
%for some $s>0$ and $\phi\in\SL(n)$.

For $\lambda\in(0,1)$, let $H$ be the hyperplane through
the origin with normal vector $(1-\lambda)e_1-\lambda e_2$,
and let
\[
\begin{aligned}
H^+&=\{x\in\mathbb R^n:
x\cdot((1-\lambda)e_1-\lambda e_2)\geq0\},\\
H^-&=\{x\in\mathbb R^n:
x\cdot((1-\lambda)e_1-\lambda e_2)\leq0\}
\end{aligned}
\]
be the two closed half-spaces bounded by $H$.
The hyperplane $H$ dissects both $T^k$ and
$\widetilde T^{k-1}$ for $2\leq k\leq n$.

Let $Z:\mathcal Q^n\to\operatorname{Sym}^p(\mathbb R^n)$
be an $\operatorname{SL}(n)$-covariant valuation.
First, consider the case $2\leq k<n$.
The inclusion--exclusion principle yields
\begin{equation}\label{eq:dissection-low}
Z(T^k)+Z(T^k\cap H)
=Z(T^k\cap H^+)+Z(T^k\cap H^-).
\end{equation}
%Let $\lambda\in(0,1)$ and denote by $H$ the hyperplane through the origin
%with normal vector $(1-\lambda)e_1-\lambda e_2$. Set
%\[
%H^+=\{x\in\mathbb R^n:x\cdot((1-\lambda)e_1-\lambda e_2)\ge0\},
%\qquad
%H^-=\{x\in\mathbb R^n:x\cdot((1-\lambda)e_1-\lambda e_2)\le0\}.
%\]
%Thus, $H^+$ and $H^-$ are the two closed half-spaces bounded by $H$.
%This hyperplane induces a series of dissections of
%$T^k$ as well as $\widetilde T^{k-1}$ for
%$k=2,\ldots,n$.
%Let $Z:\mathcal Q^n\to\Sym^p(\R^n)$ be an
%$\SL(n)$-covariant valuation.
%We distinguish two cases for these dissections. First, %assume that $k<n$.
%By the inclusion--exclusion principle, 
%\begin{equation}\label{eq:dissection-low}
%Z(T^k)+Z(T^k\cap H)
%=
%Z(T^k\cap H^+)+Z(T^k\cap H^-).
%\end{equation}

\begin{definition}\label{definition1}
Let $\lambda\in(0,1)$. The linear transformation
$\phi_1\in\SL(n)$ is given by
\[
\phi_1e_1=\lambda e_1+(1-\lambda)e_2,\quad
\phi_1e_2=e_2,\quad
\phi_1e_n=\lambda^{-1}e_n,\quad
\phi_1e_j=e_j\ \text{for }3\le j\le n-1,
\]
and $\psi_1\in\SL(n)$ is given by
\[
\psi_1e_1=e_1,\quad
\psi_1e_2=\lambda e_1+(1-\lambda)e_2,\quad
\psi_1e_n=(1-\lambda)^{-1}e_n,\quad
\psi_1e_j=e_j\ \text{for }3\le j\le n-1.
\]
\end{definition}

It is clear that $T^k\cap H^+=\psi_1T^k$,
$T^k\cap H^-=\phi_1T^k$, and
$T^k\cap H=\phi_1[o,e_1,e_3,\ldots,e_k]$.
Then, \eqref{eq:dissection-low} becomes
\[
Z(T^k)+Z(\phi_1[o,e_1,e_3,\ldots,e_k])
=
Z(\phi_1T^k)+Z(\psi_1T^k).
\]

Second, consider the dissection of  the full-dimensional 
simplex $sT^n$, where $s>0$. The inclusion--exclusion 
principle yields
\begin{equation}\label{eq:full-dissection}
Z(sT^n)+Z(sT^n\cap H)
=
Z(sT^n\cap H^+)+Z(sT^n\cap H^-).
\end{equation}

\begin{definition}\label{definition2}
Let $\lambda\in(0,1)$. The linear transformation
$\phi_2\in\GL(n)$ is given by
\[
\phi_2e_1=\lambda e_1+(1-\lambda)e_2,\quad
\phi_2e_2=e_2,\quad
\phi_2e_j=e_j\ \text{for }3\le j\le n,
\]
and $\psi_2\in\GL(n)$ is given by
\[
\psi_2e_1=e_1,\quad
\psi_2e_2=\lambda e_1+(1-\lambda)e_2,\quad
\psi_2e_j=e_j\ \text{for }3\le j\le n.
\]
\end{definition}

It is clear that $sT^n\cap H^+=\psi_2(sT^n)$,
$sT^n\cap H^-=\phi_2(sT^n)$, and
$sT^n\cap H=\phi_2(s[o,e_1,e_3,\ldots,e_n])$.
Then, \eqref{eq:full-dissection} becomes
\[
\begin{aligned}
Z(sT^n)+Z(\phi_2(s[o,e_1,e_3,\ldots,e_n]))
=Z(\phi_2(sT^n))+Z(\psi_2(sT^n)).
\end{aligned}
\]
Since $\lambda^{-1/n}\phi_2$ and
$(1-\lambda)^{-1/n}\psi_2$ belong to $\SL(n)$,
we obtain
\[
\begin{aligned}
Z(sT^n)
&+\lambda^{-p/n}\phi_2\cdot
Z(\lambda^{1/n}s[o,e_1,e_3,\ldots,e_n])\\
&=\lambda^{-p/n}\phi_2\cdot
Z(\lambda^{1/n}sT^n)
+(1-\lambda)^{-p/n}\psi_2\cdot
Z((1-\lambda)^{1/n}sT^n).
\end{aligned}
\]

The following lemmas show that $\SL(n)$-covariant symmetric
tensor valuations are uniquely determined by their values on simplices.
The proofs are similar to the  argument of
\cite[Lemmas~11 and~12]{Li2021Covariant}, or see \cite[Sections~2 and~4]{LudwigReitzner2017}
 for a similar detailed argument.

\begin{lemma}\label{lem:uniqueness-origin}
Let $Z,Z':\mathcal P_o^n\to\Sym^p(\R^n)$ be
$\SL(n)$-covariant valuations.
If $Z(sT^k)=Z'(sT^k)$ for every $s>0$ and $0\le k\le n$,
then $Z(P)=Z'(P)$ for every $P\in\mathcal P_o^n$.
\end{lemma}

\begin{lemma}\label{lem:uniqueness-all}
Let $Z,Z':\mathcal P^n\to\Sym^p(\R^n)$ be
$\SL(n)$-covariant valuations.
If $Z(sT^k)=Z'(sT^k)$ for every $s>0$ and $0\le k\le n$,
and $Z(s\widetilde T^{k-1})=Z'(s\widetilde T^{k-1})$
for every $s>0$ and $1\le k\le n$,
then $Z(P)=Z'(P)$ for every $P\in\mathcal P^n$.
\end{lemma}

\subsection{Symmetric tensors and homogeneous polynomials}
\label{sec2:symmetric-tensors}

For $v_1,\ldots,v_p\in\mathbb R^n$, the symmetric tensor
product is defined by
\[
v_1\odot\cdots\odot v_p
=\frac1{p!}\sum_{\sigma\in\mathfrak S_p}
v_{\sigma(1)}\otimes\cdots\otimes v_{\sigma(p)},
\]
where $\mathfrak S_p$ is the symmetric group on
$\{1,\ldots,p\}$. In particular, $x^{\odot p}=x^{\otimes p}$.
The natural action of $\GL(n)$ on
$\operatorname{Sym}^p(\mathbb R^n)$ is determined by 
\[
\phi\cdot(v_1\odot\cdots\odot v_p)
=\phi v_1\odot\cdots\odot\phi v_p.
\]
The Euclidean inner product induces an inner product on
$\operatorname{Sym}^p(\mathbb R^n)$ satisfying
\[
\langle x^{\odot p},u^{\odot p}\rangle=(x\cdot u)^p,
\qquad x,u\in\mathbb R^n.
\]
It is well known that there is a one-to-one correspondence
between homogeneous polynomials and symmetric tensors (see \cite{ComonEtAl2008}).
For $\mathbbm{t}\in\operatorname{Sym}^p(\mathbb R^n)$,
the corresponding polynomial is given by
\begin{equation}\label{eq1}
\widehat{\mathbbm{t}}(u)
=\langle\mathbbm{t},u^{\odot p}\rangle,
\qquad u\in\mathbb R^n.
\end{equation}
The map $\mathbbm{t}\mapsto\widehat{\mathbbm{t}}$ is injective
and satisfies
\begin{equation}\label{eq2}
\widehat{\phi\cdot\mathbbm{t}}(u)
=
\widehat{\mathbbm{t}}(\phi^{\top}u)
\end{equation}
for all $\phi\in\GL(n)$ and $u\in\mathbb R^n$,
where $\phi^{\top}$ denotes the transpose of $\phi$. 
For $m\ge1$, let
\begin{equation}\label{eq:hp}
h_p(u_1,\ldots,u_m)
=\sum_{\alpha_1+\cdots+\alpha_m=p}
u_1^{\alpha_1}\cdots u_m^{\alpha_m}
\end{equation}
denote the complete homogeneous symmetric polynomial of
degree $p$ in $m$ variables, where the sum ranges over all
nonnegative integers $\alpha_1,\ldots,\alpha_m$.

\subsection{The moment tensor}
The moment tensor $M^{p,0}$ is a simple valuation satisfying
\begin{equation}\label{eq:moment-cov}
M^{p,0}(\phi P)
=\phi\cdot M^{p,0}(P),
\qquad P\in\mathcal P^n,\quad \phi\in\SL(n).
\end{equation}
%We record the following properties of the moment tensor.
\begin{lemma}\label{lem:M}
%Let $n\ge3$ and $p\ge1$. 
For
$u=(u_1,\ldots,u_n)\in\mathbb R^n$,
\[
\widehat{M^{p,0}(T^n)}(u)
=
\frac{p!}{(n+p)!}\,h_p(u_1,\ldots,u_n).
\]
Moreover, \[
M^{p,0}(sP)=s^{n+p}M^{p,0}(P)
\]
for every $P\in\mathcal P^n$ and $s>0$.
\end{lemma}

\begin{proof}
For $\alpha=(\alpha_1,\ldots,\alpha_n)$ with nonnegative
integer entries and $\alpha_1+\cdots+\alpha_n=p$, let
$\alpha!=\alpha_1!\cdots\alpha_n!$ and
$x^\alpha=x_1^{\alpha_1}\cdots x_n^{\alpha_n}$.
The Dirichlet integral gives
\[
\int_{T^n}x^\alpha\,dx=\frac{\alpha!}{(n+p)!}.
\]
Thus, by \eqref{eq1} and the multinomial theorem,
\[
\begin{aligned}
\widehat{M^{p,0}(T^n)}(u)
&=\int_{T^n}(x\cdot u)^p\,dx\\
&=\sum_{\alpha_1+\cdots+\alpha_n=p}
\frac{p!}{\alpha!}\,u^\alpha
\int_{T^n}x^\alpha\,dx\\
&=\frac{p!}{(n+p)!}\,h_p(u_1,\ldots,u_n).
\end{aligned}
\]
For $s>0$, the change of variables $x=sy$ yields
\[
M^{p,0}(sP)
=\int_{sP}x^{\odot p}\,dx
=s^{n+p}\int_P y^{\odot p}\,dy
=s^{n+p}M^{p,0}(P).
\]
\end{proof}

\begin{lemma}\label{lem:M1}
Let $n\ge 3$ and $p\ge 1$. Define
$Z:\mathcal P^n\to\operatorname{Sym}^p(\mathbb R^n)$ by
\[
Z(P)=M^{p,0}([o,P])
\]
for every $P\in\mathcal P^n$.
Then $Z$ is an $\operatorname{SL}(n)$-covariant valuation.
\end{lemma}

\begin{proof}
Since $[o,\phi P]=\phi[o,P]$
for every \(\phi\in\mathrm{SL}(n)\), by the $\operatorname{SL}(n)$-covariance of \(M^{p,0}\),
\[
Z(\phi P)
=
M^{p,0}(\phi[o,P])
=
\phi\cdot M^{p,0}([o,P])
=
\phi\cdot Z(P).
\]
Thus \(Z\) is \(\mathrm{SL}(n)\)-covariant.

It remains to verify the valuation property. Let
\(P,Q\in\mathcal P^n\) be such that \(P\cup Q\) is convex.
We claim that
\[
[o,P\cup Q]=[o,P]\cup[o,Q],
\qquad
[o,P\cap Q]=[o,P]\cap[o,Q].
\]
The first identity and the inclusion from left to right in the second are immediate. For the reverse
inclusion, take
$x\in[o,P]\cap[o,Q]$. If $x=o$, the assertion follows from
$P\cap Q\ne\varnothing$. Otherwise, choose $\nu_1,\nu_2\ge1$ such
that $\nu_1 x\in P$ and $\nu_2x\in Q$. Since $P\cup Q$ is convex,
the nonempty closed sets $P\cap[\nu_1 x,\nu_2x]$ and $Q\cap[\nu_1 x,\nu_2x]$
cover the connected segment $[\nu_1 x,\nu_2x]$ and hence meet at
some $wx$, where $w\ge1$. Thus
$x\in[o,wx]\subseteq[o,P\cap Q]$.
The valuation property of $M^{p,0}$ now gives
\[
\begin{aligned}
Z(P\cup Q)+Z(P\cap Q)
&=M^{p,0}([o,P]\cup[o,Q])
 +M^{p,0}([o,P]\cap[o,Q])\\
&=Z(P)+Z(Q).
\end{aligned}
\]
\end{proof}

\section{$\SL(n)$-covariant valuations on $\mathcal P_o^n$}
\subsection{Values on lower-dimensional simplices}
%We first consider the standard simplices $T^k$ with $0\le k<n$.
\begin{lemma}\label{lem:point-zero}
Let $n\ge 3$ and $p\ge 1$. If $Z:\mathcal P_o^n\longrightarrow \Sym^p(\R^n)$
is an $\SL(n)$-covariant valuation, then $Z(\{o\})=\boldsymbol{0}.$
\end{lemma}

\begin{proof}
Let $\mathbbm{t}=Z(\{o\})\in\operatorname{Sym}^p(\mathbb R^n)$.
For distinct $i,j\in\{1,\ldots,n\}$ and $t\in\mathbb R$,
let $\Phi_t=I+tE_{ij}$, where $E_{ij}$ is the matrix with
$1$ in the $(i,j)$-entry and $0$ elsewhere.
Since $\Phi_t\in\SL(n)$ and $\Phi_t\{o\}=\{o\}$,
the $\SL(n)$-covariance of $Z$ gives
$\mathbbm{t}=\Phi_t\cdot\mathbbm{t}$.
By \eqref{eq2},
\[
\widehat{\mathbbm{t}}(x)
=\widehat{\mathbbm{t}}(\Phi_t^\top x)
=\widehat{\mathbbm{t}}(x+t x_i e_j)
\]
for all $x\in\mathbb R^n$ and $t\in\mathbb R$.
Differentiating with respect to $t$ at $t=0$ yields
$x_i\,\partial_j\widehat{\mathbbm{t}}(x)=0$.
Thus $\partial_j\widehat{\mathbbm{t}}$ vanishes on the
dense set $\{x\in\mathbb R^n:x_i\ne0\}$ and hence,
by continuity, on all of $\mathbb R^n$.
Since $j$ was arbitrary, $\widehat{\mathbbm{t}}$ is constant.
Its homogeneity of degree $p\ge1$ implies
$\widehat{\mathbbm{t}}=0$, and hence
$\mathbbm{t}=\boldsymbol{0}$.
\end{proof}

\begin{lemma}\label{lem3.1}
Let $n\geq 3$ and $p\geq 1$. If
$Z:\mathcal P_o^n\to\operatorname{Sym}^p(\mathbb R^n)$
is an $\mathrm{SL}(n)$-covariant valuation, then
\[
Z(T^k)\in
\operatorname{Sym}^p\bigl(\operatorname{span}\{e_1,\ldots,e_k\}\bigr),
\qquad
Z(sT^k)=s^pZ(T^k)
\]
for every $1\leq k<n$ and $s>0$.
\end{lemma}

\begin{proof}
For $k<j\le n$ and $t\in\mathbb R$, consider
$\Phi_t=I+tE_{1j}$.
Since $\Phi_t\in\SL(n)$ and $\Phi_tT^k=T^k$,
the $\SL(n)$-covariance of $Z$ gives
$Z(T^k)=\Phi_t\cdot Z(T^k)$.
By the polynomial identification,
\[
\widehat{Z(T^k)}(x)
=\widehat{Z(T^k)}(\Phi_t^\top x)
=\widehat{Z(T^k)}(x+t x_1e_j)
\]
for all $x\in\mathbb R^n$ and $t\in\mathbb R$.
Differentiating with respect to $t$ at $t=0$ yields
$x_1\partial_j\widehat{Z(T^k)}(x)=0$.
Thus $\partial_j\widehat{Z(T^k)}$ vanishes on the
dense set $\{x\in\mathbb R^n:x_1\ne0\}$ and hence,
by continuity, on all of $\mathbb R^n$.
Since $j>k$ was arbitrary, $\widehat{Z(T^k)}$ depends
only on $x_1,\ldots,x_k$. Consequently,
\[
Z(T^k)\in
\operatorname{Sym}^p\bigl(\operatorname{span}\{e_1,\ldots,e_k\}\bigr).
\]

For $s>0$, define $\varphi_s$ by
\[
\varphi_s e_i=s e_i \quad (1\le i\le k),\qquad
\varphi_s e_{k+1}=s^{-k}e_{k+1},\qquad
\varphi_s e_i=e_i \quad (k+2\le i\le n).
\]
Then $\varphi_s\in\SL(n)$ and $\varphi_sT^k=sT^k$.
Since $Z(T^k)\in
\operatorname{Sym}^p\bigl(\operatorname{span}\{e_1,\ldots,e_k\}\bigr)$,
the $\SL(n)$-covariance of $Z$ gives
\[
Z(sT^k)
=\varphi_s\cdot Z(T^k)
=s^p Z(T^k).
\]

\end{proof}

For an $\SL(n)$-covariant valuation
$Z:\mathcal P_o^n\to\Sym^p(\mathbb R^n)$ and $1\le k<n$,
define $F_k$ by
\[
F_k(x_1,\ldots,x_k)
=
\widehat{Z(T^k)}(x_1e_1+\cdots+x_ke_k).
\]
Each $F_k$ is symmetric. For $k=1$, this is immediate.
For $2\le k<n$, any transposition of $e_1,\ldots,e_k$,
combined with the change of sign of $e_{k+1}$ and fixing
the remaining basis vectors, defines an element of $\SL(n)$
that preserves $T^k$.
The assertion follows from the $\SL(n)$-covariance of $Z$
and Lemma~\ref{lem3.1}.

\begin{lemma}\label{lem:T12}
Let $n\ge3$ and $p\ge1$. If $Z:\mathcal P_o^n\to\Sym^p(\R^n)$ is an $\SL(n)$-covariant valuation, then there exists $d\in\R$ such that
\begin{equation}\label{eq:T1}
 Z(T^1)=d e_1^{\odot p}
\end{equation}
and
\begin{equation}\label{eq:T2}
 Z(T^2)=\frac d 2\bigl(e_1^{\odot p}+e_2^{\odot p}\bigr).
\end{equation}
\end{lemma}

\begin{proof}
The first assertion follows from \Cref{lem3.1}.
For $0<\lambda<1$, let
$v_\lambda=\lambda e_1+(1-\lambda)e_2$.
The dissection in Definition~\ref{definition1} with $k=2$ gives
$T^2=\phi_1T^2\cup\psi_1T^2$ and
$\phi_1T^2\cap\psi_1T^2=[o,v_\lambda]=\phi_1T^1$.
By the valuation property and the $\SL(n)$-covariance of $Z$,
\begin{equation}\label{eq11}
Z(T^2)+\phi_1\cdot Z(T^1)
=\phi_1\cdot Z(T^2)+\psi_1\cdot Z(T^2).
\end{equation}
Let $t=\lambda x_1+(1-\lambda)x_2$. Since
\[
\phi_1^\top(x_1e_1+x_2e_2)=te_1+x_2e_2,
\qquad
\psi_1^\top(x_1e_1+x_2e_2)=x_1e_1+te_2,
\]
it follows from \eqref{eq11} and \eqref{eq2} that
\begin{equation}\label{eq:F2-cocycle}
F_2(x_1,x_2)+dt^p
=F_2(t,x_2)+F_2(x_1,t).
\end{equation}
For fixed $x_1,x_2\in\mathbb R$, both sides of
\eqref{eq:F2-cocycle} are polynomials in $\lambda$.
Since the identity holds for $0<\lambda<1$,
it holds for all $\lambda\in\mathbb R$.
Taking $\lambda=1$ gives $F_2(x_1,x_1)=dx_1^p$.
For $x_1\ne x_2$, taking
$\lambda=-x_2/(x_1-x_2)$ gives $t=0$ and hence
\[
F_2(x_1,x_2)=F_2(0,x_2)+F_2(x_1,0).
\]
This identity holds on the dense set
$\{(x_1,x_2)\in\mathbb R^2:x_1\ne x_2\}$ and therefore,
by continuity, for all $x_1,x_2\in\mathbb R$.
By the homogeneity of $F_2$, there exists $c'\in\mathbb R$
such that $F_2(x_1,0)=c'x_1^p$.
The symmetry of $F_2$ then gives $F_2(0,x_2)=c'x_2^p$.
Consequently,
$F_2(x_1,x_2)=c'(x_1^p+x_2^p)$.
Taking $x_2=x_1$ and using $F_2(x_1,x_1)=dx_1^p$,
we obtain $c'=d/2$.
Thus $F_2(x_1,x_2)=\frac{d}{2}(x_1^p+x_2^p)$,
and \eqref{eq:T2} follows from \Cref{lem3.1}
and the injectivity of the polynomial identification.

\end{proof}

\begin{lemma}\label{lem:T3-zero}
Let $n\ge4$ and $p\ge1$. If $Z:\mathcal P_o^n\to\Sym^p(\R^n)$ is an $\SL(n)$-covariant valuation, then
\[
 Z(T^1)=Z(T^2)=Z(T^3)=\boldsymbol{0}.
\]
\end{lemma}

\begin{proof}
For $0<\lambda<1$, the dissection in
Definition~\ref{definition1} with $k=3$ gives
$T^3=\phi_1T^3\cup\psi_1T^3$ and
$\phi_1T^3\cap\psi_1T^3=[o,v_\lambda,e_3]$,
where $v_\lambda=\lambda e_1+(1-\lambda)e_2$.
By the valuation property of $Z$,
\[
Z(T^3)+Z([o,v_\lambda,e_3])
=Z(\phi_1T^3)+Z(\psi_1T^3).
\]
By \eqref{eq:T2} and the $\SL(n)$-covariance of $Z$,
\[
Z([o,v_\lambda,e_3])
=\frac d2\left(v_\lambda^{\odot p}
+e_3^{\odot p}\right).
\]
Thus, by the $\SL(n)$-covariance of $Z$ and \eqref{eq2},
\begin{equation}\label{eq:F3-cocycle}
F_3(x_1,x_2,x_3)
+\frac d2(t^p+x_3^p)
=F_3(t,x_2,x_3)+F_3(x_1,t,x_3),
\end{equation}
where $t=\lambda x_1+(1-\lambda)x_2$.
For fixed $x_1,x_2,x_3\in\mathbb R$,
both sides of \eqref{eq:F3-cocycle} are polynomials in $\lambda$.
Since the identity holds for $0<\lambda<1$,
it holds for all $\lambda\in\mathbb R$.
Taking $\lambda=1$ gives
\begin{equation}\label{eq:F3-diag}
F_3(x_1,x_1,x_3)
=\frac d2(x_1^p+x_3^p).
\end{equation}
For $x_1\ne x_2$, taking
$\lambda=(x_3-x_2)/(x_1-x_2)$ gives $t=x_3$.
By \eqref{eq:F3-cocycle}, \eqref{eq:F3-diag},
and the symmetry of $F_3$,
\[
\begin{aligned}
F_3(x_1,x_2,x_3)+dx_3^p
&=F_3(x_3,x_2,x_3)+F_3(x_1,x_3,x_3)\\
&=\frac d2(x_3^p+x_2^p)
+\frac d2(x_3^p+x_1^p).
\end{aligned}
\]
Hence $F_3(x_1,x_2,x_3)=\frac d2(x_1^p+x_2^p)$
whenever $x_1\ne x_2$.
This identity holds on the dense set
$\{(x_1,x_2,x_3)\in\mathbb R^3:x_1\ne x_2\}$
and therefore, by continuity, on all of $\mathbb R^3$.
The symmetry of $F_3$ now gives $d=0$.
Consequently, $F_3\equiv0$, and \Cref{lem3.1}
and the injectivity of the polynomial identification imply
$Z(T^3)=\boldsymbol{0}$.
By \Cref{lem:T12}, we also have
$Z(T^1)=Z(T^2)=\boldsymbol{0}$.

\end{proof}

\begin{lemma}\label{lem:lower-induction}
Let $4\le k<n$ and $p\ge1$. If $Z:\mathcal P_o^n\to\Sym^p(\R^n)$ is an $\SL(n)$-covariant valuation and $Z(T^{k-1})=\boldsymbol{0}$, then $Z(T^k)=\boldsymbol{0}$.
\end{lemma}

\begin{proof}
Apply the dissection from Definition~\ref{definition1}
with $k$. Then
$T^k=\phi_1T^k\cup\psi_1T^k,$ $
\phi_1T^k\cap\psi_1T^k
=[o,v_\lambda,e_3,\ldots,e_k]$,
where $v_\lambda=\lambda e_1+(1-\lambda)e_2$.
Since $k<n$, the linear map sending
$e_1,\ldots,e_{k-1}$ to
$v_\lambda,e_3,\ldots,e_k$, respectively,
extends to an element of $\SL(n)$.
Thus $\phi_1T^k\cap\psi_1T^k$ is an
$\SL(n)$-image of $T^{k-1}$.
By the induction hypothesis and the
$\SL(n)$-covariance of $Z$,
\[
Z(\phi_1T^k\cap\psi_1T^k)=\boldsymbol{0}.
\]
The valuation property and the $\SL(n)$-covariance of $Z$ imply
\[
Z(T^k)=\phi_1\cdot Z(T^k)+\psi_1\cdot Z(T^k).
\]
Since
\[
\begin{aligned}
\phi_1^\top\left(\sum_{j=1}^k x_je_j\right)
&=t e_1+\sum_{j=2}^k x_je_j,\qquad
\psi_1^\top\left(\sum_{j=1}^k x_je_j\right)
=x_1e_1+t e_2+\sum_{j=3}^k x_je_j,
\end{aligned}
\]
where
$t=\lambda x_1+(1-\lambda)x_2,$
the polynomial identification gives
\begin{equation}\label{eq:Fk-functional}
\begin{aligned}
F_k(x_1,x_2,x_3,\ldots,x_k)
={}&F_k(t,x_2,x_3,\ldots,x_k)
+F_k(x_1,t,x_3,\ldots,x_k).
\end{aligned}
\end{equation}
For fixed $x_1,\ldots,x_k\in\mathbb R$,
\eqref{eq:Fk-functional} is a polynomial identity in $\lambda$
that holds on $(0,1)$ and hence on all of $\mathbb R$.
Taking $\lambda=1$ gives
\begin{equation}\label{eq:Fk-diagonal}
F_k(x_1,x_1,x_3,\ldots,x_k)=0.
\end{equation}
For $x_1\ne x_2$, choose
$\lambda=(x_3-x_2)/(x_1-x_2)$, so that $t=x_3$.
Then \eqref{eq:Fk-functional} yields
\[
F_k(x_1,x_2,x_3,\ldots,x_k)
=
F_k(x_3,x_2,x_3,\ldots,x_k)
+
F_k(x_1,x_3,x_3,\ldots,x_k).
\]
By the symmetry of $F_k$ and \eqref{eq:Fk-diagonal},
both terms on the right-hand side vanish. Hence
\[
F_k(x_1,x_2,x_3,\ldots,x_k)=0
\]
whenever $x_1\ne x_2$.
Since $\{x\in\mathbb R^k:x_1\ne x_2\}$ is dense in
$\mathbb R^k$ and $F_k$ is continuous, we obtain
$F_k\equiv0$. Consequently,
$Z(T^k)=\boldsymbol{0}$.
\end{proof}

\subsection{A polynomial characterization}

For a permutation matrix $\pi$, denote by $\pi(i)$ the index
such that $\pi e_i=e_{\pi(i)}$.
Since $\pi(sT^n)=sT^n$ for every $s>0$,
\eqref{eq:epsilon-cov} yields
\begin{equation}\label{eq:perm-parity-tensor}
\pi\cdot Z^\varepsilon(sT^n)
=
(\det\pi)^\varepsilon Z^\varepsilon(sT^n).
\end{equation}
For $r>0$ and $\varepsilon\in\{0,1\}$, define
\begin{equation}\label{eq:Adef}
G_r^\varepsilon(x)
=r^{-p/n}\widehat{Z^\varepsilon(r^{1/n}T^n)}(x),
\qquad x\in\mathbb R^n .
\end{equation}
Assume that $Z^\varepsilon$ is simple.
Let $a,b>0$ and put $r=a+b$ and $\lambda=a/r$.
Applying the dissection in Definition~\ref{definition2} to
$r^{1/n}T^n$ and using the simplicity of $Z^\varepsilon$, we obtain
\[
Z^\varepsilon(r^{1/n}T^n)
=
Z^\varepsilon(r^{1/n}\phi_2T^n)
+
Z^\varepsilon(r^{1/n}\psi_2T^n).
\]
Since
\[
r^{1/n}\phi_2T^n
=(\lambda^{-1/n}\phi_2)(a^{1/n}T^n),
\qquad
r^{1/n}\psi_2T^n
=((1-\lambda)^{-1/n}\psi_2)(b^{1/n}T^n),
\]
while $\lambda^{-1/n}\phi_2,
(1-\lambda)^{-1/n}\psi_2\in \mathrm{SL}(n),$
\eqref{eq:epsilon-cov} implies
\[
\begin{aligned}
Z^\varepsilon(r^{1/n}T^n)
={}&
\lambda^{-p/n}\phi_2\cdot
Z^\varepsilon(a^{1/n}T^n)
+
(1-\lambda)^{-p/n}\psi_2\cdot
Z^\varepsilon(b^{1/n}T^n).
\end{aligned}
\]
Passing to the associated polynomials and multiplying
by $r^{-p/n}$, we obtain from \eqref{eq2}
\begin{equation}\label{eq:pure-cocycle}
\begin{aligned}
G_{a+b}^\varepsilon(x_1,x_2,x_3,\ldots,x_n)
&=G_a^\varepsilon(t,x_2,x_3,\ldots,x_n)
+G_b^\varepsilon(x_1,t,x_3,\ldots,x_n),
\end{aligned}
\end{equation}
where $t=(ax_1+bx_2)/(a+b)$.
Moreover, \eqref{eq:perm-parity-tensor} and \eqref{eq2} imply
\begin{equation}\label{eq:perm-parity-poly}
G_r^\varepsilon(x_{\pi(1)},\ldots,x_{\pi(n)})
=(\det\pi)^\varepsilon G_r^\varepsilon(x_1,\ldots,x_n)
\end{equation}
for every permutation matrix $\pi$.

\begin{lemma}\label{lem:endpoint}
Let $p\ge1$ and fix $\varepsilon\in\{0,1\}$.
For each $r>0$, let $G_r^\varepsilon$ be a homogeneous
polynomial of degree $p$ satisfying
\begin{equation}\label{eq:two-variable-cocycle}
G_{a+b}^\varepsilon(x_1,x_2)
=
G_a^\varepsilon\left(\frac{ax_1+bx_2}{a+b},x_2\right)
+
G_b^\varepsilon\left(x_1,\frac{ax_1+bx_2}{a+b}\right)
\end{equation}
for all $a,b>0$ and $x_1,x_2\in\R$.
Suppose that
$G_{a+b}^\varepsilon(0,1)
=
G_a^\varepsilon(0,1)+G_b^\varepsilon(0,1)$
for all $a,b>0$.
Then there exists $c_0\in\R$ such that
\[
G_r^\varepsilon(x_1,x_2)
=
c_0r\,h_p(x_1,x_2),
\qquad
h_p(x_1,x_2)=\sum_{j=0}^{p}x_1^j x_2^{p-j},
\]
for all $r>0$ and $x_1,x_2\in\R$.
\end{lemma}

\begin{proof}
Let $D(r)=G_r^\varepsilon(0,1)$.
Evaluating \eqref{eq:two-variable-cocycle} at $(0,1)$
and using the homogeneity of $G_b^\varepsilon$
and the additivity of $D$, we obtain
\begin{equation}\label{eq:endpoint-reduction1}
G_a^\varepsilon(\xi,1)
=
D(a)+(1-\xi^p)D(b),
\qquad
\xi=\frac{b}{a+b},
\end{equation}
for all $a,b>0$.
Taking $a=1$ and $b=r$ yields
\[
D(r)=
\frac{
G_1^\varepsilon\!\left(\frac{r}{1+r},1\right)-D(1)
}{
1-\left(\frac{r}{1+r}\right)^p
},
\qquad r>0.
\]
Since $G_1^\varepsilon$ is a polynomial and the denominator
is positive, $D$ is continuous on $(0,\infty)$.
By the additivity of $D$,
\[
m_1D\left(\frac {m_2} {m_1}\right)=D(m_2)=m_2D(1)
\]
for all positive integers $m_1,m_2$.
Thus $D(q)=qD(1)$ for every positive rational $q$.
By continuity and the density of the positive rationals
in $(0,\infty)$, this identity extends to all $r>0$.
Hence $D(r)=c_0r$ for all $r>0$, where $c_0=D(1)$.
For fixed $r>0$ and $0<\xi<1$, substituting
$a=r$ and $b=r\xi/(1-\xi)$ into
\eqref{eq:endpoint-reduction1} yields
\[
G_r^\varepsilon(\xi,1)
=
c_0r\left(1+\frac{\xi(1-\xi^p)}{1-\xi}\right)
=
c_0r\,h_p(\xi,1).
\]
Since both sides are polynomials in $\xi$ and agree
on $(0,1)$, the identity holds for all $\xi\in\mathbb R$.
By the homogeneity of $G_r^\varepsilon$ and $h_p$,
we obtain
\[
G_r^\varepsilon(x_1,x_2)
=
c_0r\,h_p(x_1,x_2),
\qquad x_1,x_2\in\mathbb R.
\]
Since $r>0$ was arbitrary, the proof is complete.
\end{proof}
%The moment tensor is a simple $\SL(n)$-covariant valuation.
By \Cref{lem:M},
\[
r^{-p/n}\widehat{M^{p,0}(r^{1/n}T^n)}(x)
=
\frac{p!}{(n+p)!}\,r h_p(x).
\]
Applying the same dissection argument as in
\eqref{eq:pure-cocycle} to $M^{p,0}$ gives
\begin{equation}\label{eq:hp-identity}
(a+b)h_p(x_1,x_2,x')
=
a h_p(t,x_2,x')
+
b h_p(x_1,t,x'),
\qquad
t=\frac{ax_1+bx_2}{a+b}.
\end{equation}
Here $x'=(x_3,\ldots,x_n)$.

\begin{theorem}\label{thm:poly-rigidity}
Let $n\ge3$, $p\ge1$, and $\varepsilon\in\{0,1\}$.
For each $r>0$, let $G_r^\varepsilon$ be a homogeneous polynomial of degree $p$
on $\R^n$. Suppose that
\begin{equation}\label{eq:rigid-identity}
G_{a+b}^\varepsilon(x_1,x_2,x')
=G_a^\varepsilon(t,x_2,x')+G_b^\varepsilon(x_1,t,x'),
\qquad t=\frac{ax_1+bx_2}{a+b},
\end{equation}
for all $a,b>0$ and $x\in\R^n$ and that
\begin{equation}\label{eq:rigid-parity}
G_r^\varepsilon(x_{\pi(1)},\ldots,x_{\pi(n)})
=(\det\pi)^\varepsilon G_r^\varepsilon(x_1,\ldots,x_n)
\end{equation}
for every permutation matrix $\pi$.
If $\varepsilon=0$, then $G_r^0=c_0rh_p$
for some $c_0\in\R$ and all $r>0$. If $\varepsilon=1$, then
$G_r^1\equiv0$
for all $r>0$.
\end{theorem}

\begin{proof}
We first show that the restrictions to $x_n=0$ determine a family
uniquely for each fixed $\varepsilon$.
Let $\{f_r\}_{r>0}$ be a family of homogeneous polynomials of degree $p$
satisfying \eqref{eq:rigid-identity} and \eqref{eq:rigid-parity}
with $G_r^\varepsilon$ replaced by $f_r$, and suppose that
$f_r(x_1,\ldots,x_{n-1},0)=0$ for every $r>0$.
By \eqref{eq:rigid-parity}, $f_r$ vanishes whenever any coordinate is zero.
Setting $x_1=0$ in \eqref{eq:rigid-identity} therefore gives
\[
f_a\left(\frac{b}{a+b}x_2,x_2,x'\right)=0.
\]
For fixed $a>0$, the ratio $b/(a+b)$ ranges over $(0,1)$.
Thus $f_a$ vanishes on the nonempty open set
$\{x\in\mathbb R^n:x_2\ne0,\ 0<x_1/x_2<1\}$,
and hence $f_a\equiv0$.
Since \eqref{eq:rigid-identity} and \eqref{eq:rigid-parity} are linear
in the family, the difference of any two families with the same
$\varepsilon$ and the same restrictions to $x_n=0$ satisfies
these conditions. By the preceding argument, the two families coincide.
This proves uniqueness.

We now prove the assertion by induction on $n$.
First consider the case $n=3$.
The restrictions $G_r^\varepsilon(x_1,x_2,0)$ satisfy
\eqref{eq:two-variable-cocycle}.
Evaluating \eqref{eq:rigid-identity} at $(0,0,1)$ gives
\[
G_{a+b}^\varepsilon(0,0,1)
=
G_a^\varepsilon(0,0,1)
+
G_b^\varepsilon(0,0,1),
\]
so the map
$r\mapsto G_r^\varepsilon(0,0,1)$ is additive.
Applying \eqref{eq:rigid-parity} to the transposition of the
second and third coordinates yields
\[
G_r^\varepsilon(0,0,1)
=
(-1)^\varepsilon G_r^\varepsilon(0,1,0).
\]
Hence $r\mapsto G_r^\varepsilon(0,1,0)$ is additive.
Applying \Cref{lem:endpoint} to
$G_r^\varepsilon(x_1,x_2,0)$, we obtain
\[
G_r^\varepsilon(x_1,x_2,0)=c_0r\,h_p(x_1,x_2)
\]
for some $c\in\mathbb R$.

If $\varepsilon=1$, then
\eqref{eq:rigid-parity} gives
$G_r^\varepsilon(x_2,x_1,0)
=
-G_r^\varepsilon(x_1,x_2,0).$
Since $h_p$ is symmetric, it follows that $c_0=0$.
Consequently, $G_r^\varepsilon(x_1,x_2,0)=0$
for all $x_1,x_2\in\mathbb R$.
The uniqueness established above then implies
$G_r^\varepsilon\equiv0$.

If $\varepsilon=0$, then by \eqref{eq:hp-identity},
the family $G_r^\varepsilon-c_0rh_p$ satisfies
\eqref{eq:rigid-identity} and \eqref{eq:rigid-parity}.
Moreover, its restriction to $x_3=0$ vanishes.
The uniqueness result therefore gives
$G_r^\varepsilon=c_0rh_p.$

Now let $n\ge4$ and assume that the assertion holds in dimension
$n-1$.
The restrictions
$G_r^\varepsilon(x_1,\ldots,x_{n-1},0)$ satisfy the same
assumptions in dimension $n-1$.
Hence, by the induction hypothesis, they equal
$c_0rh_p(x_1,\ldots,x_{n-1})$ when $\varepsilon=0$,
and vanish when $\varepsilon=1$.
Since $h_p(x_1,\ldots,x_{n-1},0)
=
h_p(x_1,\ldots,x_{n-1}),$
the uniqueness result completes the induction.
\end{proof}
\subsection{The case $n=3$}\label{sec:three-dimensional}

\begin{lemma}\label{lem:n3-lower-parity}

Let $p\ge1$. If
$Z:\mathcal P_o^3\to\Sym^p(\R^3)$ is an
$\SL(3)$-covariant valuation, then
\[
Z^0(P)=Z(P),\qquad Z^1(P)=\boldsymbol{0}
\]
for every $P\in\mathcal P_o^3$ with $\dim P<3$.
Moreover, there is a constant $d\in\R$ such that
\[
Z^0(T^1)=d e_1^{\odot p},\qquad
Z^0(T^2)=\frac d2(e_1^{\odot p}+e_2^{\odot p}).
\]
\end{lemma}

\begin{proof}
The case $P=\{o\}$ follows from \Cref{lem:point-zero}.
Otherwise, set $L=\operatorname{span}P$. 
Applying the argument in the proof of \Cref{lem3.1}
after an $\SL(3)$ change of coordinates, we obtain
$Z(P)\in\Sym^p(L)$.
Choose $\theta\in\SL^{\pm}(3)\setminus\SL(3)$ as in
\eqref{eq:parity}.
Since $\dim L<3$, there exists
$\widetilde\theta\in\SL(3)$ such that
$\widetilde\theta|_L=\theta^{-1}|_L$.
Thus $\widetilde\theta P=\theta^{-1}P$, and the
$\SL(3)$-covariance of $Z$ gives
\[
\theta\cdot Z(\theta^{-1}P)
=\theta\widetilde\theta\cdot Z(P)
=Z(P),
\]
where the last equality follows since
$\theta\widetilde\theta$ fixes $L$ pointwise and
$Z(P)\in\Sym^p(L)$.
By \eqref{eq:parity},
$Z^0(P)=Z(P)$ and $Z^1(P)=\boldsymbol{0}$.
The remaining identities follow from
\eqref{eq:T1} and \eqref{eq:T2}.
\end{proof}

By \Cref{lem:n3-lower-parity}, $Z^1$ is simple.
Hence $G_r^1$ satisfies
\eqref{eq:pure-cocycle} and \eqref{eq:perm-parity-poly}.
Applying \Cref{thm:poly-rigidity}, we obtain
$G_r^1\equiv0$ for all $r>0$.
By \eqref{eq:Adef} and the injectivity of the polynomial
identification, $Z^1(r^{1/3}T^3)=\boldsymbol{0}$ for $r>0.$
Replacing $r^{1/3}$ by $s$ gives
\begin{equation}\label{eq:n3-oddzero}
Z^1(sT^3)=\boldsymbol{0},\qquad s>0.
\end{equation}
The same conclusion holds for every $T\in\mathcal T^3$.
Indeed, let $T=[o,v_1,v_2,v_3]$ and, after interchanging
two vertices if necessary, assume that
$\det[v_1\ v_2\ v_3]>0$.
Define
\[
s=(\det[v_1\ v_2\ v_3])^{1/3},
\qquad
\Phi_1=s^{-1}[v_1\ v_2\ v_3].
\]
Then $\Phi_1\in\SL(3)$ and $T=\Phi_1(sT^3)$.
By \eqref{eq:epsilon-cov} and \eqref{eq:n3-oddzero},
\[
Z^1(T)=\Phi_1\cdot Z^1(sT^3)=\boldsymbol{0}.
\]

We next consider $Z^0$ and the polynomials $G_r^0$ defined in
\eqref{eq:Adef}, with $n=3$.
For $a,b>0$, let $\lambda=a/(a+b)$,
$s=(a+b)^{1/3}$, and $t=(ax_1+bx_2)/(a+b)$.
By Definition~\ref{definition2},
\[
T^3=\phi_2T^3\cup\psi_2T^3,
\qquad
\phi_2T^3\cap\psi_2T^3
=[o,\lambda e_1+(1-\lambda)e_2,e_3].
\]
By \Cref{lem3.1,lem:n3-lower-parity},
\[
s^{-p}\widehat{Z^0(s[o,\lambda e_1+(1-\lambda)e_2,e_3])}(x)
=\frac d2(t^p+x_3^p).
\]
Since $\det\phi_2=\lambda$ and
$\det\psi_2=1-\lambda$, the maps
$\lambda^{-1/3}\phi_2$ and
$(1-\lambda)^{-1/3}\psi_2$ belong to $\SL(3)$.
By the $\SL^\pm(3)$-$0$-covariance of $Z^0$
and \eqref{eq:Adef},
\[
\begin{aligned}
s^{-p}\widehat{Z^0(s\phi_2T^3)}(x)
&=G_a^0(t,x_2,x_3),\qquad
s^{-p}\widehat{Z^0(s\psi_2T^3)}(x)
=G_b^0(x_1,t,x_3).
\end{aligned}
\]
Thus the valuation property of $Z^0$ and
\eqref{eq2} yield
\begin{equation}\label{eq:n3-identity}
G_{a+b}^0(x_1,x_2,x_3)+\frac d2(t^p+x_3^p)
=G_a^0(t,x_2,x_3)+G_b^0(x_1,t,x_3).
\end{equation}
We now show that $d=0$, so that
\eqref{eq:n3-identity} reduces to
\eqref{eq:pure-cocycle}.

\begin{lemma}\label{lem:n3-defect-zero}
Let $p\ge1$. If
$Z:\mathcal P_o^3\to\Sym^p(\R^3)$
is an $\SL(3)$-covariant valuation, then
\[
Z^0(T^1)=Z^0(T^2)=\boldsymbol{0}.
\]
\end{lemma}

\begin{proof}

Let $d$ be as in \Cref{lem:n3-lower-parity}, and define
$g(r)=G_r^0(1,1,1)$ for $r>0$.
Taking $(x_1,x_2,x_3)=(1,1,1)$ in
\eqref{eq:n3-identity} gives
\begin{equation}\label{eq:g-additive}
g(a+b)+d=g(a)+g(b).
\end{equation}

For $u\in\R$, substitute
$x_1=1+u$, $x_2=1-u$, and $x_3=1$ into
\eqref{eq:n3-identity}. Then
$t=1+(a-b)u/(a+b)$.
By \eqref{eq:perm-parity-poly}, $G_r^0$ is symmetric.
Since $G_r^0$ is homogeneous of degree $p$,
\[
G_r^0(1+u,1+u,1+u)=(1+u)^p g(r).
\]
Differentiating \eqref{eq:n3-identity} with respect to
$u$ at $u=0$ and using
$\partial_iG_r^0(1,1,1)=\frac{p}{3}g(r)$ for $i=1,2,3$,
we obtain
\[
a\left(g(b)-\frac{3d}{4}\right)
=
b\left(g(a)-\frac{3d}{4}\right),
\qquad a,b>0.
\]
Taking $b=1$ gives
\[
g(r)=c'r+\frac{3d}{4},\qquad r>0,
\]
where $c'=g(1)-\frac{3d}{4}$.
Substituting this relation into \eqref{eq:g-additive} yields
$d=0$.
\end{proof}

\begin{lemma}\label{lem:n3-even-simple}
Let $p\ge1$. If
$Z:\mathcal P_o^3\to\Sym^p(\R^3)$
is an $\SL(3)$-covariant valuation, then $Z$ is simple.
\end{lemma}

\begin{proof}
By \Cref{lem:n3-defect-zero}, $d=0$.
Hence \eqref{eq:n3-identity} reduces to
\eqref{eq:pure-cocycle} for $G_r^0$.
Together with \eqref{eq:perm-parity-poly},
\Cref{thm:poly-rigidity} yields
\begin{equation}\label{eq:n3-even-poly}
G_r^0(x_1,x_2,x_3)
=c_0rh_p(x_1,x_2,x_3),\qquad r>0,
\end{equation}
for some $c_0\in\R$.

By \Cref{lem:point-zero,lem:n3-defect-zero,lem:n3-lower-parity},
$Z(T^k)=\boldsymbol{0}$ for $0\le k<3$.
Every $T\in\mathcal T^k$, $k=1,2$, is an
$\SL(3)$-image of $T^k$.
Thus, by the $\SL(3)$-covariance of $Z$,
$Z(T)=\boldsymbol{0}$.
We now prove that $Z$ vanishes on all polytopes of
 dimension less than $3$. The assertion for dimension $0$
 follows from \Cref{lem:point-zero}. Assume that
$Z(Q)=\boldsymbol{0}$
for every $Q\in\mathcal P_o^3$ with
$\dim Q<k$,
where $k\in\{1,2\}$. Let $P\in\mathcal P_o^3$ have dimension $k$.
Triangulate $P$ into simplices
$T_1,\ldots,T_m\in\mathcal T^k$.
 Every nonempty intersection of two or more distinct simplices
 has dimension less than $k$ and contains the origin.
By the induction hypothesis, $Z$ vanishes on each such intersection.
 Since $Z(T_i)=\boldsymbol{0}$, \Cref{lemma2.1} gives $Z(P)=\boldsymbol{0}$.
 Hence $Z$ is simple.
\end{proof}

\begin{lemma}\label{prop:simplex-n3}
Let $p\ge1$. If
$Z:\mathcal P_o^3\to\Sym^p(\R^3)$
is an $\SL(3)$-covariant valuation, then there exists
$c\in\R$ such that
\[
Z(sT^3)=cM^{p,0}(sT^3),\qquad s>0.
\]
\end{lemma}

\begin{proof}
By \eqref{eq:n3-oddzero},
$Z^1(sT^3)=\boldsymbol{0}$ for all $s>0$.
By \eqref{eq:Adef}, \eqref{eq:n3-even-poly}, and \Cref{lem:M},
\[
r^{-p/3}\widehat{Z^0(r^{1/3}T^3)}(x)
=cr\widehat{M^{p,0}(T^3)}(x)
\]
for all $r>0$ and $x\in\mathbb R^3$, where
$c=c_0(p+3)!/p!$ and $c_0$ denotes the constant in
\eqref{eq:n3-even-poly}.
Taking $r=s^3$ and using the injectivity of the polynomial
identification and the homogeneity of $M^{p,0}$, we obtain
\[
Z(sT^3)=Z^0(sT^3)
=cs^{p+3}M^{p,0}(T^3)
=cM^{p,0}(sT^3).
\]
\end{proof}

\subsection{The case $n\ge4$}

\begin{lemma}\label{prop:simple-n4}
Let $n\ge4$ and $p\ge1$. If $Z:\mathcal P_o^n\to\Sym^p(\R^n)$ is an $\SL(n)$-covariant valuation, then $Z$ is simple.
\end{lemma}

\begin{proof}
By \Cref{lem:point-zero,lem:T3-zero,lem:lower-induction},
we have $Z(T^k)=\boldsymbol{0}$ for $0\le k<n$.
We first show that $Z(T)=\boldsymbol{0}$ for every
$T\in\mathcal T^k$ with $1\le k<n$.
Write $T=[o,v_1,\ldots,v_k]$.
Since $v_1,\ldots,v_k$ are linearly independent
and $k<n$, there exists $\phi\in\SL(n)$ such that
$\phi e_i=v_i$ for $1\le i\le k$.
Thus $T=\phi T^k$, and the $\SL(n)$-covariance of $Z$ gives
\begin{equation}\label{eq:Z-T-zero}
Z(T)=\phi\cdot Z(T^k)=\boldsymbol{0}.
\end{equation}

We now proceed by induction on $k=\dim P$.
The case $k=0$ follows from \Cref{lem:point-zero}.
Suppose that $Z(Q)=\boldsymbol{0}$ for every
$Q\in\mathcal P_o^n$ with $\dim Q\le k-1$.
Let $P\in\mathcal P_o^n$ with $\dim P=k<n$.
Choose a triangulation
$P=T_1\cup\cdots\cup T_m$
such that $T_1,\ldots,T_m\in\mathcal T^k$.
By \eqref{eq:Z-T-zero},
$Z(T_i)=\boldsymbol{0}$ for $1\le i\le m$.
For $2\le j\le m$ and
$1\le i_1<\cdots<i_j\le m$, the triangulation gives
\[
o\in T_{i_1}\cap\cdots\cap T_{i_j},
\qquad
\dim(T_{i_1}\cap\cdots\cap T_{i_j})\le k-1.
\]
Hence, by the induction hypothesis,
$Z(T_{i_1}\cap\cdots\cap T_{i_j})=\boldsymbol{0}$.
Applying \Cref{lemma2.1}, we obtain
\[
\begin{aligned}
Z(P)
&=\sum_{i=1}^m Z(T_i)
+\sum_{j=2}^m(-1)^{j-1}
\sum_{1\le i_1<\cdots<i_j\le m}
Z(T_{i_1}\cap\cdots\cap T_{i_j})\\
&=\sum_{i=1}^m Z(T_i)
=\boldsymbol{0}.
\end{aligned}
\]
Thus $Z$ is simple.
\end{proof}

%The three-dimensional case requires a separate argument, given in \Cref{sec:three-dimensional}.

\begin{lemma}\label{prop:simplex-n4}
Let $n\ge4$ and $p\ge1$. If
$Z:\mathcal P_o^n\to\Sym^p(\R^n)$
is an $\SL(n)$-covariant valuation, then there exists $c\in\R$ such that
\[
Z(sT^n)=c\,M^{p,0}(sT^n)
\]
for every $s>0$.
\end{lemma}

\begin{proof}
By \Cref{prop:simple-n4}, $Z$ is simple, which implies
that $Z^\varepsilon$ is simple for $\varepsilon\in\{0,1\}$.
The polynomials $G_r^\varepsilon$ defined in \eqref{eq:Adef}
satisfy \eqref{eq:rigid-identity} and
\eqref{eq:rigid-parity}.
Thus, by \Cref{thm:poly-rigidity}, there exists
$c_0\in\R$ such that
\[
G_r^0(x)=c_0r\,h_p(x),
\qquad
G_r^1\equiv0,
\qquad r>0.
\]
By \Cref{lem:M},
\[
\widehat{M^{p,0}(T^n)}(x)
=
\frac{p!}{(n+p)!}\,h_p(x).
\]
Taking $r=s^n$ in \eqref{eq:Adef} and using the injectivity
of the polynomial identification, we obtain
\[
Z^0(sT^n)=cs^{n+p}M^{p,0}(T^n),
\qquad
Z^1(sT^n)=\boldsymbol{0},
\]
where $c=c_0(n+p)!/p!$.
Since $Z=Z^0+Z^1$, the homogeneity of $M^{p,0}$ gives
\[
Z(sT^n)=cs^{n+p}M^{p,0}(T^n)
=cM^{p,0}(sT^n),
\qquad s>0.
\]
\end{proof}

Finally, we obtain the following classification.
\begin{proof}[Proof of \Cref{thm:main}]
For the ``if'' direction, let $Z(P)=cM^{p,0}(P) =c\int_P x^{\odot p}\,dx .$
For $P,Q,P\cup Q\in\mathcal P_o^n$, we have
\[
\begin{aligned}
Z(P\cup Q)+Z(P\cap Q)
&=
c\int_{P\cup Q}x^{\odot p}\,dx
+c\int_{P\cap Q}x^{\odot p}\,dx\\
&=
c\int_{P\setminus Q}x^{\odot p}\,dx
+c\int_{Q\setminus P}x^{\odot p}\,dx
+2c\int_{P\cap Q}x^{\odot p}\,dx\\
&=
c\int_Px^{\odot p}\,dx
+c\int_Qx^{\odot p}\,dx\\
&=
Z(P)+Z(Q).
\end{aligned}
\]
Thus $Z$ is a valuation. Moreover, for $\phi\in\SL(n)$,
the change of variables $x=\phi y$ gives
\[
\begin{aligned}
Z(\phi P)
=
c\int_{\phi P}x^{\odot p}\,dx
=
c\int_P(\phi y)^{\odot p}|\det\phi|\,dy
=
c\,\phi\cdot\int_Py^{\odot p}\,dy
=
\phi\cdot Z(P).
\end{aligned}
\]
Hence $Z$ is $\SL(n)$-covariant.

Conversely, let
$Z:\mathcal P_o^n\to\Sym^p(\R^n)$
be an $\SL(n)$-covariant valuation. By \Cref{prop:simple-n4,prop:simplex-n4} when $n\ge4$ and by
\Cref{lem:n3-even-simple,prop:simplex-n3} when $n=3$, the valuation
$Z$ is simple and there exists $c\in\R$ such that
\[
Z(sT^n)=cM^{p,0}(sT^n),
\qquad s>0.
\]

Let $T=[o,v_1,\ldots,v_n]\in\mathcal T^n$. 
Since interchanging two nonzero vertices preserves the simplex,
we may assume that
$\det[v_1\ \cdots\ v_n]>0$. Set
\[
s=\bigl(\det[v_1\ \cdots\ v_n]\bigr)^{1/n},
\qquad
\Phi_2=s^{-1}[v_1\ \cdots\ v_n].
\]
Then $\Phi_2\in\SL(n)$ and
$T=\Phi_2(sT^n)$. Hence, by the
$\SL(n)$-covariance of $Z$ and $M^{p,0}$,
\[
Z(T)
=
\Phi_2\cdot Z(sT^n)
=
c\,\Phi_2\cdot M^{p,0}(sT^n)
=
cM^{p,0}(T).
\]

Finally, let $P\in\mathcal P_o^n$ be full-dimensional.
Triangulate $P$ into simplices
$T_1,\ldots,T_m\in\mathcal T^n$. Every nonempty intersection of two
or more distinct $T_i$ has dimension less than $n$. Since both $Z$ and $M^{p,0}$ are
simple, all these intersection terms vanish. 
By \Cref{lemma2.1},
\[
\begin{aligned}
Z(P)
&=
\sum_{i=1}^m Z(T_i)
+
\sum_{j=2}^m(-1)^{j-1}
\sum_{1\le i_1<\cdots<i_j\le m}
Z(T_{i_1}\cap\cdots\cap T_{i_j})\\
&=
\sum_{i=1}^m Z(T_i)
=
c\sum_{i=1}^m M^{p,0}(T_i)
=
cM^{p,0}(P).
\end{aligned}
\]
If $\dim P<n$, both sides vanish. Hence \eqref{eq:main-class} follows.
\end{proof}
\section{$\SL(n)$-covariant valuations on $\mathcal P^n$}\label{sec:all-polytopes}

Let
$Z:\mathcal P^n\to\Sym^p(\mathbb R^n)$
be an $\SL(n)$-covariant valuation, and let $Z^\varepsilon$ be defined by
\eqref{eq:parity}. For $r>0$ and $\varepsilon\in\{0,1\}$, set
\begin{equation}\label{eq:tilde-G}
\widetilde G_r^\varepsilon(x)
=
r^{-p/n}
\widehat{Z^\varepsilon(r^{1/n}\widetilde T^{n-1})}(x).
\end{equation}
Since $\widetilde T^{n-1}$ is invariant under coordinate permutations,
\eqref{eq:epsilon-cov} and \eqref{eq2} imply that
\begin{equation}\label{eq:tilde-permutation}
\widetilde G_r^\varepsilon
(x_{\pi(1)},\ldots,x_{\pi(n)})
=
(\det\pi)^\varepsilon
\widetilde G_r^\varepsilon(x_1,\ldots,x_n)
\end{equation}
for every permutation matrix $\pi$.

\begin{lemma}\label{lem:tilde-low}
Let $n\ge3$ and $p\ge1$. If
$Z:\mathcal P^n\to\Sym^p(\R^n)$ is an $\SL(n)$-covariant valuation,
then $Z(T)=\boldsymbol{0}$ for every
$T\in\widetilde{\mathcal T}^{k-1}$ with $1\le k\le n-1$.
\end{lemma}

\begin{proof}
By \eqref{eq:tilde-simplex-normalize}, it suffices to prove that
$Z(s\widetilde T^{m-1})=\boldsymbol{0}$ for
$1\le m<n$ and $s>0$.
For $1\le m<n$, the maps $\Phi_t=I+tE_{1j}$,
where $m<j\le n$ and $t\in\R$, fix
$\widetilde T^{m-1}$ pointwise.
As in the proof of \Cref{lem3.1},
\begin{equation}\label{eq:tilde-support}
Z(\widetilde T^{m-1})
\in
\Sym^p\bigl(\operatorname{span}\{e_1,\ldots,e_m\}\bigr).
\end{equation}
For $s>0$, let $\varphi_s$ be the map defined in the proof
of \Cref{lem3.1}, with $k=m$.
Since $\varphi_s\widetilde T^{m-1}
=s\widetilde T^{m-1}$, the $\SL(n)$-covariance
of $Z$ and \eqref{eq:tilde-support} yield
\begin{equation}\label{eq:tilde-scale}
Z(s\widetilde T^{m-1})
=
\varphi_s\cdot Z(\widetilde T^{m-1})
=
s^p Z(\widetilde T^{m-1}).
\end{equation}

For $1\le m<n$, define
\[
\widetilde F_m(x_1,\ldots,x_m)
=
\widehat{Z(\widetilde T^{m-1})}
(x_1e_1+\cdots+x_me_m).
\]
Then $\widetilde F_m$ is homogeneous of degree $p$.
For $m\ge2$, every permutation of $e_1,\ldots,e_m$
extends to an element of $\SL(n)$ that preserves
$\widetilde T^{m-1}$.
Thus the $\SL(n)$-covariance of $Z$ implies that
$\widetilde F_m$ is symmetric.

For $2\le m<n$ and $0<\lambda<1$, let
$v_\lambda=\lambda e_1+(1-\lambda)e_2$.
By Definition~\ref{definition1},
\[
\widetilde T^{m-1}
=
[v_\lambda,e_2,e_3,\ldots,e_m]
\cup
[e_1,v_\lambda,e_3,\ldots,e_m],
\]
with intersection $[v_\lambda,e_3,\ldots,e_m]$.
Since $m<n$, the linearly independent vectors
$v_\lambda,e_3,\ldots,e_m$ can be completed
to a basis of $\R^n$ with determinant $1$.
Thus the intersection is an $\SL(n)$-image
of $\widetilde T^{m-2}$.
By the valuation property and the
$\SL(n)$-covariance of $Z$,
\begin{equation}\label{eq:tilde-recursion}
\widetilde F_m(x_1,x_2,x')
+
\widetilde F_{m-1}(t,x')
=
\widetilde F_m(t,x_2,x')
+
\widetilde F_m(x_1,t,x'),
\end{equation}
where $x'=(x_3,\ldots,x_m)$ and
$t=\lambda x_1+(1-\lambda)x_2$.

Since $\widetilde F_1$ is homogeneous of degree $p$,
there exists $d\in\R$ such that
$\widetilde F_1(x_1)=dx_1^p$.
For $m=2$, \eqref{eq:tilde-recursion} gives
\[
\widetilde F_2(x_1,x_2)+dt^p
=
\widetilde F_2(t,x_2)+\widetilde F_2(x_1,t),
\qquad
t=\lambda x_1+(1-\lambda)x_2.
\]
For fixed $x_1,x_2\in\R$, both sides are
polynomials in $\lambda$.
Since they agree on $(0,1)$, the identity holds
for all $\lambda\in\R$.
Taking $\lambda=1$ gives
$\widetilde F_2(x_1,x_1)=dx_1^p$.
For $x_1\ne x_2$, choose
$\lambda=-x_2/(x_1-x_2)$, so that $t=0$.
Then
\[
\widetilde F_2(x_1,x_2)
=
\widetilde F_2(0,x_2)+\widetilde F_2(x_1,0).
\]
By the homogeneity and symmetry of $\widetilde F_2$,
there exists $c\in\R$ such that
$\widetilde F_2(x_1,x_2)=c(x_1^p+x_2^p)$
whenever $x_1\ne x_2$.
Since both sides are continuous, this identity
holds for all $x_1,x_2\in\R$.
Taking $x_1=x_2$ gives $c=d/2$. Hence
\begin{equation}\label{eq:tilde-n3-low}
\widetilde F_1(x_1)=dx_1^p,
\qquad
\widetilde F_2(x_1,x_2)
=
\frac d2(x_1^p+x_2^p).
\end{equation}

Suppose first that $n\ge4$.
For $m=3$, \eqref{eq:tilde-recursion}
and \eqref{eq:tilde-n3-low} yield
\[
\widetilde F_3(x_1,x_2,x_3)
+\frac d2(t^p+x_3^p)
=
\widetilde F_3(t,x_2,x_3)
+
\widetilde F_3(x_1,t,x_3),
\]
where $t=\lambda x_1+(1-\lambda)x_2$.
Since both sides are polynomials in $\lambda$ and agree
on $(0,1)$, the identity holds for all $\lambda\in\R$.
Taking $\lambda=1$ gives
\[
\widetilde F_3(x_1,x_1,x_3)
=
\frac d2(x_1^p+x_3^p).
\]
For $x_1\ne x_2$, choose
$\lambda=(x_3-x_2)/(x_1-x_2)$, so that $t=x_3$.
By the symmetry of $\widetilde F_3$, we obtain
\[
\widetilde F_3(x_1,x_2,x_3)
=
\frac d2(x_1^p+x_2^p).
\]
Since both sides are polynomials, this identity
holds for all $x_1,x_2,x_3\in\R$.
The symmetry of $\widetilde F_3$ then implies
$d=0$. Consequently,
$\widetilde F_1\equiv\widetilde F_2
\equiv\widetilde F_3\equiv0$.

For $4\le m<n$, assume inductively that
$\widetilde F_{m-1}\equiv0$.
Then \eqref{eq:tilde-recursion} reduces to
\[
\widetilde F_m(x_1,x_2,x')
=
\widetilde F_m(t,x_2,x')
+
\widetilde F_m(x_1,t,x'),
\]
where $t=\lambda x_1+(1-\lambda)x_2$.
Since both sides are polynomials in $\lambda$
and agree on $(0,1)$, the identity holds for
all $\lambda\in\R$.
Taking $\lambda=1$ gives
$\widetilde F_m(x_1,x_1,x')=0$.
For $x_1\ne x_2$, choose
$\lambda=(x_3-x_2)/(x_1-x_2)$, so that $t=x_3$.
By the symmetry of $\widetilde F_m$, both terms
on the right-hand side vanish. Hence
$\widetilde F_m(x_1,x_2,x')=0$
whenever $x_1\ne x_2$.
Since $\widetilde F_m$ is continuous,
we obtain $\widetilde F_m\equiv0$.
Thus, by induction,
$\widetilde F_m\equiv0$ for $1\le m<n$.
The injectivity of the polynomial identification
and \eqref{eq:tilde-scale} complete the proof
for $n\ge4$.

It remains to consider $n=3$.
Let $E=\operatorname{span}\{e_1,e_2\}$.
By \eqref{eq:tilde-support},
$Z(\widetilde T^1)\in\Sym^p(E)$.
Choose $\theta\in\SL^{\pm}(3)\setminus\SL(3)$
as in \eqref{eq:parity}.
Since $\dim E<3$, there exists
$\widetilde{\theta}\in\SL(3)$ such that
$\widetilde{\theta}|_E=\theta^{-1}|_E$.
Thus
$\widetilde{\theta}\widetilde T^1
=\theta^{-1}\widetilde T^1$, and the
$\SL(3)$-covariance of $Z$ gives
\[
Z(\theta^{-1}\widetilde T^1)
=
\widetilde{\theta}\cdot Z(\widetilde T^1).
\]
Since $\theta\widetilde{\theta}$ fixes $E$
pointwise and $Z(\widetilde T^1)\in\Sym^p(E)$,
we have
$\theta\cdot Z(\theta^{-1}\widetilde T^1)
=Z(\widetilde T^1)$.
Hence, by \eqref{eq:parity} and
\eqref{eq:tilde-n3-low},
\[
Z^0(\widetilde T^1)
=
Z(\widetilde T^1)
=
\frac d2\bigl(e_1^{\odot p}+e_2^{\odot p}\bigr).
\]

For $a,b>0$, let $s=(a+b)^{1/3}$ and
$\lambda=a/(a+b)$.
Taking $v_\lambda=\lambda e_1+(1-\lambda)e_2$,
we have
\[
s\widetilde T^2
=
s[v_\lambda,e_2,e_3]
\cup s[e_1,v_\lambda,e_3],
\]
with intersection $s[v_\lambda,e_3]$.
Define $\psi_3$ by
$\psi_3e_1=sv_\lambda$,
$\psi_3e_2=se_3$, and
$\psi_3e_3=(s/b)e_1$.
Since $\det\psi_3=s^3(1-\lambda)/b=1$,
we have $\psi_3\in\SL(3)$ and
$\psi_3\widetilde T^1=s[v_\lambda,e_3]$.
By the $\SL(3)$-covariance of $Z^0$,
\[
\begin{aligned}
Z^0(s[v_\lambda,e_3])
&=\psi_3\cdot Z^0(\widetilde T^1)
=\frac d2s^p
\bigl(v_\lambda^{\odot p}+e_3^{\odot p}\bigr).
\end{aligned}
\]

By Definition~\ref{definition2},
$\det\phi_2=\lambda$ and
$\det\psi_2=1-\lambda$.
Thus $\lambda^{-1/3}\phi_2$ and
$(1-\lambda)^{-1/3}\psi_2$ belong to $\SL(3)$.
Moreover,
\[
\begin{aligned}
s[v_\lambda,e_2,e_3]
&=(\lambda^{-1/3}\phi_2)
(a^{1/3}\widetilde T^2),\qquad
s[e_1,v_\lambda,e_3]
=((1-\lambda)^{-1/3}\psi_2)
(b^{1/3}\widetilde T^2).
\end{aligned}
\]
By the valuation property and the
$\SL(3)$-covariance of $Z^0$,
\[
\begin{aligned}
Z^0(s\widetilde T^2)
+\frac d2s^p
\bigl(v_\lambda^{\odot p}+e_3^{\odot p}\bigr)
={}&
(\lambda^{-1/3}\phi_2)\cdot
Z^0(a^{1/3}\widetilde T^2)
+
((1-\lambda)^{-1/3}\psi_2)\cdot
Z^0(b^{1/3}\widetilde T^2).
\end{aligned}
\]
Taking the associated polynomials and multiplying
by $(a+b)^{-p/3}$ yields
\begin{equation}\label{eq:tilde-n3-identity}
\widetilde G_{a+b}^0(x_1,x_2,x_3)
+\frac d2(t^p+x_3^p)
=
\widetilde G_a^0(t,x_2,x_3)
+
\widetilde G_b^0(x_1,t,x_3),
\end{equation}
where $t=(ax_1+bx_2)/(a+b)$.

For $r>0$, define
$g(r)=\widetilde G_r^0(1,1,1)$.
Evaluating \eqref{eq:tilde-n3-identity}
at $(1,1,1)$ gives
\begin{equation}\label{eq:tilde-diagonal}
g(a+b)+d=g(a)+g(b).
\end{equation}
By \eqref{eq:tilde-permutation},
$\widetilde G_r^0$ is symmetric.
Since it is homogeneous of degree $p$,
\[
\widetilde G_r^0(1+u,1+u,1+u)
=
(1+u)^p g(r).
\]
Differentiating at $u=0$ and using the symmetry
of $\widetilde G_r^0$, we obtain
$\partial_i\widetilde G_r^0(1,1,1)
=
\frac p3g(r)$ for $i=1,2,3.$
Substituting $x_1=1+u$, $x_2=1-u$,
and $x_3=1$ into \eqref{eq:tilde-n3-identity},
and differentiating at $u=0$, we obtain
\begin{equation}\label{eq:tilde-n3-derivative}
\frac{dp}{2}\frac{a-b}{a+b}
=
\frac{2p}{3(a+b)}
\bigl(ag(b)-bg(a)\bigr).
\end{equation}
Taking $a=b=1$ in \eqref{eq:tilde-diagonal}
and $a=2$, $b=1$ in
\eqref{eq:tilde-n3-derivative} gives
\[
2g(1)-g(2)=d,
\qquad
2g(1)-g(2)=\frac{3d}{4}.
\]
Thus $d=0$.
By the injectivity of the polynomial identification,
\eqref{eq:tilde-scale} and
\eqref{eq:tilde-n3-low},
\[
Z(s\widetilde T^0)
=
Z(s\widetilde T^1)
=
\boldsymbol{0},
\qquad s>0.
\]
\end{proof}

%We next extend the preceding vanishing result to all polytopes of dimension at most $n-2$.

\begin{lemma}\label{prop:residual-low}
Let $n\ge3$ and $p\ge1$. If
$Z:\mathcal P^n\to\Sym^p(\R^n)$ is an $\SL(n)$-covariant valuation,
then $Z(P)=\boldsymbol{0}$
for every $P\in\mathcal P^n$ with $\dim P\le n-2$.
\end{lemma}

\begin{proof}
By \Cref{thm:main},
$Z(Q)=\boldsymbol{0}$ for every
$Q\in\mathcal P_o^n$ with $\dim Q<n$.
It therefore suffices to consider
$P\in\mathcal P^n\setminus\mathcal P_o^n$.
We argue by induction on $k=\dim P$.
For $k=0$, the assertion follows from \Cref{lem:tilde-low}.
Let $1\le k\le n-2$ and assume that
$Z(Q)=\boldsymbol{0}$ for every
$Q\in\mathcal P^n\setminus\mathcal P_o^n$
with $\dim Q\le k-1$.
Let $P\in\mathcal P^n\setminus\mathcal P_o^n$
have dimension $k$.

First, suppose that $o\notin\operatorname{aff}P$. Triangulate $P$ into $k$-dimensional
simplices $T_1,\ldots,T_N$. Then
$T_i\in\widetilde{\mathcal T}^{k}$ for every $i$, while every
nonempty intersection of distinct simplices has dimension at most
$k-1$. By \Cref{lemma2.1}, the induction hypothesis, and
\Cref{lem:tilde-low}, we obtain $Z(P)=\boldsymbol{0}.$

Second, suppose that $o\in\operatorname{aff}P\setminus P$.
Let $E_1,\ldots,E_m$ be the facets of $P$ visible from
the origin, namely those whose supporting hyperplanes
in $\operatorname{aff}P$ strictly separate $o$ from
$\operatorname{relint}P$.
Choose compatible triangulations of these facets into
$(k-1)$-dimensional simplices
$\widetilde T_1,\ldots,\widetilde T_\ell$.
Then the simplices $T_i=[o,\widetilde T_i]$,
$1\le i\le\ell$, triangulate the closure of
$[o,P]\setminus P$, and
\[
[o,P]=P\cup T_1\cup\cdots\cup T_\ell,
\qquad
P\cap T_i=\widetilde T_i.
\]
Every intersection of the $T_i$ contains the origin
and has dimension at most $k$, whereas every nonempty
intersection $T_{i_1}\cap\cdots\cap T_{i_j}\cap P$
has dimension at most $k-1$.
Hence, by \Cref{lemma2.1}, \Cref{thm:main}, and the induction
hypothesis,
\[
\begin{aligned}
\boldsymbol{0}
=Z([o,P])
&=
\sum_{j=1}^{\ell}(-1)^{j-1}
\sum_{1\le i_1<\cdots<i_j\le\ell}
Z(T_{i_1}\cap\cdots\cap T_{i_j})\\
&\quad+
\sum_{j=1}^{\ell}(-1)^j
\sum_{1\le i_1<\cdots<i_j\le\ell}
Z(T_{i_1}\cap\cdots\cap T_{i_j}\cap P)
+Z(P)\\
&=Z(P).
\end{aligned}
\]
This completes the induction.
\end{proof}
%The remaining lower-dimensional values are determined by one constant.

\begin{lemma}\label{prop:facet-simplex}
Let $n\ge3$ and $p\ge1$. If
$Z:\mathcal P^n\to\Sym^p(\R^n)$ is an $\SL(n)$-covariant valuation, then
there exists $\gamma\in\R$ such that
\begin{equation}\label{eq:Z-codimension-one}
Z(P)=\gamma M^{p,0}([o,P])
\end{equation}
for every $(n-1)$-dimensional polytope $P\in\mathcal P^n$.
\end{lemma}

\begin{proof}
First, we determine $Z(s\widetilde T^{n-1})$ for $s>0$.
Let $a,b>0$, and set
$r=a+b$, $s=r^{1/n}$, $\lambda=a/r$, and
$v_\lambda=\lambda e_1+(1-\lambda)e_2$.
Fix $\varepsilon\in\{0,1\}$.
The dissection in Definition~\ref{definition2} gives
\[
s\widetilde T^{n-1}
=
s[v_\lambda,e_2,\ldots,e_n]
\cup
s[e_1,v_\lambda,e_3,\ldots,e_n],
\]
with intersection $s[v_\lambda,e_3,\ldots,e_n]$.
Since this intersection has dimension $n-2$,
\Cref{prop:residual-low}, applied to $Z^\varepsilon$, gives
\[
Z^\varepsilon(s[v_\lambda,e_3,\ldots,e_n])
=\boldsymbol{0}.
\]

By Definition~\ref{definition2},
$\det\phi_2=\lambda$ and $\det\psi_2=1-\lambda$.
Hence $\lambda^{-1/n}\phi_2$ and
$(1-\lambda)^{-1/n}\psi_2$ belong to $\SL(n)$.
Since $a=\lambda r$ and $b=(1-\lambda)r$, we have
\[
s[v_\lambda,e_2,\ldots,e_n]
=
(\lambda^{-1/n}\phi_2)
(a^{1/n}\widetilde T^{n-1})
\]
and
\[
s[e_1,v_\lambda,e_3,\ldots,e_n]
=
((1-\lambda)^{-1/n}\psi_2)
(b^{1/n}\widetilde T^{n-1}).
\]
Thus, by the valuation property and the
$\SL(n)$-covariance of $Z^\varepsilon$,
\[
\begin{aligned}
Z^\varepsilon(s\widetilde T^{n-1})
={}&
(\lambda^{-1/n}\phi_2)\cdot
Z^\varepsilon(a^{1/n}\widetilde T^{n-1})
+
((1-\lambda)^{-1/n}\psi_2)\cdot
Z^\varepsilon(b^{1/n}\widetilde T^{n-1}).
\end{aligned}
\]
Multiplying the associated polynomial identity by $r^{-p/n}$
and using $a=\lambda r$ and $b=(1-\lambda)r$, we obtain
\begin{equation}\label{eq:tilde-pure-cocycle}
\widetilde G_{a+b}^\varepsilon(x_1,x_2,x')
=
\widetilde G_a^\varepsilon(t,x_2,x')
+
\widetilde G_b^\varepsilon(x_1,t,x'),
\end{equation}
where $x'=(x_3,\ldots,x_n)$ and
$t=(ax_1+bx_2)/(a+b)$.
Together with \eqref{eq:tilde-permutation},
\Cref{thm:poly-rigidity} gives
$\widetilde G_r^1\equiv0$ and
$\widetilde G_r^0=\gamma_0rh_p$
for some $\gamma_0\in\R$ and all $r>0$.
By \Cref{lem:M}, setting
$\gamma=\gamma_0(n+p)!/p!$ and $r=s^n$, we obtain
\begin{equation}\label{eq:facet-simplex}
Z(s\widetilde T^{n-1})
=
\gamma M^{p,0}(sT^n)
=
\gamma M^{p,0}([o,s\widetilde T^{n-1}]),
\qquad s>0.
\end{equation}

Second, let $P$ be an $(n-1)$-dimensional polytope with
$o\notin\operatorname{aff}P$.
Triangulate $P$ into
$T_1,\ldots,T_N\in\widetilde{\mathcal T}^{n-1}$.
By \eqref{eq:tilde-simplex-normalize},
\eqref{eq:facet-simplex}, and the $\SL(n)$-covariance
of $Z$ and $M^{p,0}$,
\[
Z(T_i)=\gamma M^{p,0}([o,T_i]),
\qquad 1\le i\le N.
\]
The simplices $[o,T_1],\ldots,[o,T_N]$ triangulate $[o,P]$.
Hence, by the inclusion-exclusion principle,
\Cref{prop:residual-low}, and the simplicity of $M^{p,0}$,
\[
Z(P)
=
\sum_{i=1}^N Z(T_i)
=
\gamma\sum_{i=1}^N M^{p,0}([o,T_i])
=
\gamma M^{p,0}([o,P]).
\]

Finally, let $P$ be an $(n-1)$-dimensional polytope with
$o\in\operatorname{aff}P$.
Then $\dim[o,P]=n-1$, and hence
$M^{p,0}([o,P])=\boldsymbol{0}$.
If $o\in P$, then \Cref{thm:main} gives
$Z(P)=\boldsymbol{0}$.
Assume that $o\notin P$, and let
$E_1,\ldots,E_m$ be the facets of $P$ visible from the origin.
Then
\[
[o,P]
=
P\cup[o,E_1]\cup\cdots\cup[o,E_m],
\qquad
P\cap[o,E_i]=E_i.
\]
Every nonempty intersection involving only the sets $[o,E_i]$
contains the origin and has dimension less than $n$, while every
nonempty intersection involving $P$ and at least one $[o,E_i]$
has dimension at most $n-2$.
Hence, by
\Cref{prop:simple-n4,lem:n3-even-simple},
\Cref{prop:residual-low}, and the inclusion-exclusion principle,
\[
\boldsymbol{0}
=
Z([o,P])
=
Z(P)-\sum_{i=1}^m Z(E_i)
=
Z(P).
\]
\end{proof}
Finally, we establish the classification in \Cref{thm:all}.
\begin{proof}[Proof of \Cref{thm:all}]
By \eqref{eq:moment-cov} and \Cref{lem:M1}, every map of the form
\eqref{eq:all-class} is an $\SL(n)$-covariant valuation. It remains
to prove the converse.

Let $Z:\mathcal P^n\to\Sym^p(\R^n)$ be an $\SL(n)$-covariant
valuation. Let $c$ be the constant from \Cref{thm:main} and
$\gamma$ the constant from \Cref{prop:facet-simplex}. By \Cref{thm:main}, and by
\Cref{prop:residual-low,prop:facet-simplex} in the lower-dimensional cases,
\eqref{eq:all-class} holds with $c_1=c-\gamma$ and $c_2=\gamma$ whenever
$o\in P$ or $\dim P<n$. Thus it remains to consider
$P\in\mathcal P^n\setminus\mathcal P_o^n$ with $\dim P=n$.

Let $E_1,\ldots,E_m$ be the facets of $P$ visible from the origin.
Then
\[
[o,P]=P\cup[o,E_1]\cup\cdots\cup[o,E_m],
\qquad
P\cap[o,E_i]=E_i.
\]
Every intersection of two or more of the sets $[o,E_i]$ contains
the origin and has dimension at most $n-1$, while every nonempty
intersection of $P$ with two or more of these sets has dimension
at most $n-2$. Hence, by
\Cref{thm:main}, \Cref{prop:residual-low}
and the inclusion--exclusion principle,
\[
\begin{aligned}
cM^{p,0}([o,P])
=Z([o,P])
&=
\sum_{j=2}^{m}(-1)^{j-1}
\sum_{1\le i_1<\cdots<i_j\le m}
Z\bigl([o,E_{i_1}]\cap\cdots\cap[o,E_{i_j}]\bigr)
\\
&\quad+
\sum_{j=2}^{m}(-1)^j
\sum_{1\le i_1<\cdots<i_j\le m}
Z\bigl(P\cap[o,E_{i_1}]\cap\cdots\cap[o,E_{i_j}]\bigr)
\\
&\quad-\sum_{i=1}^m Z(P\cap[o,E_i])
+\sum_{i=1}^m Z([o,E_i])+Z(P)
\\
&=
\sum_{i=1}^m Z([o,E_i])
+Z(P)-\sum_{i=1}^m Z(E_i)
\\
&=
c\sum_{i=1}^m M^{p,0}([o,E_i])
+Z(P)
-\gamma\sum_{i=1}^m M^{p,0}([o,E_i]).
\end{aligned}
\]
Since $M^{p,0}$ is simple, the same decomposition gives
\[
M^{p,0}([o,P])
=
M^{p,0}(P)
+\sum_{i=1}^m M^{p,0}([o,E_i]).
\]
Substituting this identity into the preceding equality yields
\[
Z(P)
=
(c-\gamma)M^{p,0}(P)
+\gamma M^{p,0}([o,P]).
\]
Thus \eqref{eq:all-class} holds with
$c_1=c-\gamma$ and $c_2=\gamma$, which completes the proof.
\end{proof}

\end{document}